\documentclass[12pt, reqno]{amsart} 
\usepackage{amsmath,amssymb,amsthm}
\usepackage{verbatim}
\usepackage{libertine}
\usepackage{enumitem}
\usepackage{bm} 
\usepackage[libertine]{newtxmath}
\usepackage[colorlinks, urlcolor=blue,  citecolor=purple]{hyperref}
\usepackage[all]{xy}%
\usepackage[utf8]{inputenc}
\usepackage{xcolor}
\usepackage{tikz-cd}

\usepackage{mathabx} 

\def\setminus{\mathchoice
	{\mathbin{\vrule height .92ex width 1.81ex depth -.58ex}}
	{\mathbin{\vrule height .92ex width 1.81ex depth -.58ex}}
	{\mathbin{\vrule height .65ex width 1.00ex depth -.43ex}}
	{\mathbin{\vrule height .50ex width 0.770ex depth -.34ex}}
}

\newcommand{\bmone}{\bm{\mathrm{1}}}
\newcommand{\bma}{\bm{\mathrm{a}}}
\newcommand{\bmA}{\bm{\mathrm{A}}}

\newcommand{\bmD}{\bm{\mathrm{D}}}

\newcommand{\bmG}{\bm{\mathrm{G}}}
\newcommand{\bmh}{\bm{\mathrm{h}}}
\newcommand{\bmH}{\bm{\mathrm{H}}}

\newcommand{\bmN}{\bm{\mathrm{N}}}
\newcommand{\bmO}{\bm{\mathrm{O}}}

\newcommand{\bmP}{\bm{\mathrm{P}}}

\newcommand{\bmr}{\bm{\mathrm{r}}}

\newcommand{\bms}{\bm{\mathrm{s}}}

\newcommand{\bmX}{\bm{\mathrm{X}}}

\newcommand{\bmeta}{\bm{\eta}}
\newcommand{\bmgamma}{\bm{\gamma}}

\newcommand{\bmtheta}{\bm{\theta}}

\newcommand{\calB}{\mathcal{B}}

\newcommand{\calD}{\mathcal{D}}
\newcommand{\calF}{\mathcal{F}}
\newcommand{\calG}{\mathcal{G}}

\newcommand{\calL}{\mathcal{L}}

\newcommand{\calN}{\mathcal{N}}
\newcommand{\calO}{\mathcal{O}}

\newcommand{\calS}{\mathcal{S}}

\newcommand{\calU}{\mathcal{U}}

\newcommand{\diff}{\mathrm{d}}  

\newcommand{\fraka}{\mathfrak{a}}

\newcommand{\frakK}{\mathfrak{K}}

\newcommand{\overbmX}{\overline{ 
\bm{\mathrm{X}}
}}

\newcommand{\rmG}{\mathrm{G}}

\newcommand{\rmm}{\mathrm{m}}

\newcommand{\rmU}{\mathrm{U}}

\newcommand{\scrC}{\mathscr{C}}
\newcommand{\scrD}{\mathscr{D}}

\newcommand{\scrH}{\mathscr{H}}
\newcommand{\scrI}{\mathscr{I}}
\newcommand{\scrL}{\mathscr{L}}

\newcommand{\wtB}{\widetilde{B}}

\newcommand{\Z}{\mathbb{Z}}
\newcommand{\Q}{\mathbb{Q}}
\newcommand{\R}{\mathbb{R}}
\newcommand{\C}{\mathbb{C}}

\DeclareMathOperator{\SL}{ {\bm{\mathrm{SL} }   }   }

\DeclareMathOperator{\SU}{ {\bm{\mathrm{SU} }   }   }
\DeclareMathOperator{\SO}{ {\bm{\mathrm{SO} }   }   }

\newcommand{\ep}{\varepsilon}

\newcommand{\normm}[1]{\left\lvert#1\right\rvert}
\newcommand{\norm}[1]{\left\lVert#1\right\rVert}

\newcommand{\bs}{\backslash}
\newcommand{\la}{\langle}
\newcommand{\ra}{\rangle}
\newcommand{\midd}{\;\middle\vert\;}
\newcommand{\ideal}{\triangleleft}

\DeclareMathOperator{\diag}{\mathrm{diag}}

\DeclareMathOperator{\Hom}{\mathrm{Hom}}

\DeclareMathOperator{\id}{\mathrm{id}}

\DeclareMathOperator{\Lip}{\mathrm{Lip}}

\DeclareMathOperator{\Nm}{\mathrm{Nm}}

\DeclareMathOperator{\Prob}{\mathrm{Prob}}

\DeclareMathOperator{\supp}{\mathrm{supp}}

\DeclareMathOperator{\tr}{\mathrm{tr}}

\DeclareMathOperator{\Vol}{\mathrm{Vol}}

\newtheorem{thm}{Theorem}[section]

\newtheorem{lem}[thm]{Lemma}
\newtheorem{prop}[thm]{Proposition}

\newtheorem{rmk}[thm]{Remark}

\theoremstyle{plain}

\theoremstyle{definition}

\newtheorem*{remark*}{Remark}
\newtheorem*{remarks*}{Remarks}

\usepackage{geometry}
\begin{document}

\title[]{Effective count of integer points on ternary affine quadrics and effective equidistribution}
\author[]{Runlin Zhang}

\address{College of Mathematics and Statistics, Center of Mathematics, Chongqing University, 401331, Chongqing,  China}
\email{runlinzhang@cqu.edu.cn}


\maketitle

\begin{abstract}
We study the effective equidistribution of certain infinite homogeneous measures and related counting problems through mixing. In this way, we obtain smooth versions of counting theorems studied by Oh--Shah and later by Kelmer–Kontorovich over a number field. In the appendix, we apply the meromorphic continuation of Hilbert–Asai Eisenstein series to obtain the authentic counting. 
\end{abstract}


\section{Introduction}

Given a variety with infinitely many integral points, a fundamental problem is to obtain an asymptotic count for points of bounded height.
For homogeneous varieties, a general counting strategy was proposed by Duke--Rudnick--Sarnak \cite{DukRudSar93}, which reduces the counting problem to an equidistribution problem on a different but related homogeneous space.
Concretely, if $\bmG/\bmH$ is the variety under consideration and $\Gamma$ is an arithmetic lattice, then 
the auxiliary homogeneous space is $\bmG(\R)/\Gamma$ and the equidistribution problem is to study the limiting distribution of homogeneous measures supported on $g\bmH(\R)\Gamma/\Gamma$ as $g$ varies. To derive the desired count, one aims to show that, ``most of the time'', the limiting measure is the full $\bmG(\R)$-invariant Haar measure.

Dynamical methods for this problem were introduced by \cite{EskMcM93, EskMozSha96}. 
It was understood that any limiting measure must be homogeneous. Thus when the $\bmH$ is maximal in  $\bmG$, the problem has essentially been solved.
However, when intermediate groups exist between $\bmH$ and $\bmG$, the limit could be ``focused'' on such a proper intermediate group.
When the homogeneous measure on $\bmH(\R)\Gamma/\Gamma$ is finite, this focusing phenomenon can be analyzed by studying certain equivariant compactifications of $\bmG/\bmH$ by \cite{zhangrunlin_2024}. A natural next step is to generalize this approach to the setting of infinite homogeneous measures.

From another viewpoint, one hopes not only an asymptotic count but also a polynomially effective error term, both for its intrinsic interest and for potential applications. When $\bmH$ is symmetric with no nontrivial $\Q$-characters,  extending the work of \cite{EskMcM93}, an effective count has been obtained by \cite{BenOh12} using a thickening trick together with effective mixing.  Note that the absence of nontrivial $\Q$-characters is equivalent to the homogeneous measure on $\bmH(\R)\Gamma/\Gamma$ being finite. 
A natural goal is to drop this finiteness assumption. A first step in this direction was taken by \cite{OhSha14} where they consider 
ternary affine quadrics $q(x)=m$ where $q$ is an integral ternary quadratic form and $m$ is a nonzero integer such that $-m\det(q) $ is a square.
The homogeneous space in question is $\SL_2$ modulo a maximal $\Q$-split torus.
 Using effective mixing, they obtained a count with logarithmic error.
  Subsequently, 
 \cite{KelKon18, Kelmer_Kontorovich_2020} obtained a polynomially effective count 
 using Eisenstein series, revealing a second order term\footnote{Independently, Oh and Shah obtain a smooth version of this using a deep result of Strombergersson \cite{Strombergsson_2004}.}. 
 Huang \cite{Huang_quadrics_I}  also obtained the second order term using Heath-Brown's delta method, albeit with a logarithmic rather than a polynomial error.

This naturally raises the question: can effective mixing alone, without appealing to Eisenstein series or the spectral theory of automorphic forms, yield a polynomially effective count? Coincidentally, the relevant homogeneous measures are infinite. It therefore also serves as natural testing grounds for generalizing the framework of \cite{zhangrunlin_2024}. With these two motivations, we revisit this problem in the present paper. As a mild generalization, we work with a general number fields.

It turns out that effective mixing alone suffices to establish a smooth version of the counting result. However, upgrading this to an authentic count requires showing that certain distributions are in fact continuous measures. Unfortunately, this lies beyond the scope of effective mixing.

Finally, we note that  this approach is promising to tackle counting problem on a much wider class of homogeneous varieties, particularly in light of the recent progress in effective Ratner theorems \cite{Lindenstrauss_Mohammadi_Wang_effective, Yanglei_2024}. 
One may hope that an effective version of \cite{EskMozSha96} could be deduced from effective Ratner theorem (see \cite{Sarkar_2025} for recent advance). 
If so, the ideas developed in this paper should be instrumental in converting such effective equidistribution results into effective counting statements.

\subsection{Notations}

Throughout the paper, we adopt the following notation:
\begin{itemize}
    \item    $k$ is a number field of degree $l$ over $\Q$. 
     Assume that $k$ has $l_1$ real embeddings  and  $l_2$ pairs of complex embeddings 
     \[
        \{\sigma_{1},...,\sigma_{l_1}\} \sqcup \{\sigma_{l_1+1},... , \sigma_{l_1+ l_2} \} \sqcup \{ \overline{\sigma}_{l_1+1},... , \overline {\sigma}_{l_1+ l_2} \}.
     \]
     So $l= l_1 + 2l_2$.
    We also let $\nu_1,...,\nu_{l_1+l_2}$ and $k_{\nu_1},...,k_{\nu_{l_1+l_2}}$ be the corresponding valuations (identified with the pull back of $\sigma_{l_i}$'s) and local fields (isomorphic to either $\R$ or $\C$).
    The product of all infinite places $k_{\infty}:= k \otimes_{\Q} \R$ is then identified with
    $\prod_{i=1}^{l_1+l_2} k_{\nu_i} \cong \R^{ l_1} \oplus \C^{ l_2}$. 
    Let $\calO_k$ be its ring of integers,
        \item Given a Lie group and a lattice $\Gamma$ of $G$, we fix a right invariant metric $d_G$ on $G$ and let $d_Y$ denote the quotient metric on $Y:=G/\Gamma$.
   \item Given a positive number $T$, let $\bmO(T)$ denote a real number whose absolute value is $\leq T$. No implicit constants are implied for this notation.
   On the other hand $O(T)$ is the usual big $O$ notation and the dependence of the implicit constants will be specified in the context.
\end{itemize}

Let $q$ be a nondegenerate ternary quadratic form over $k$ and  $m\in k^{\times}$ be such that $-m \det{q} \in (k^{\times})^2$.
Let $\bmX$ be the affine quadric defined by $\{q(x_1,x_2,x_3)=m\} \subset \bmA_k^3$ .
$\bmX$ is then a homogeneous space under the action $\mathrm{Spin}_q \cong \bmG:= \SL_2$. For each $x\in \bmX(k)$, the stabilizer of $x$ in $\bmG$ is a maximal $k$-split torus since $q$ restricted to the orthogonal complement of $x$ is $k$-isotropic. 
We  assume that $ \bmX(k)\neq \emptyset $ and let $\Gamma $ be an arithmetic subgroup of $\bmG$.

\subsection{Counting}

Fix some Euclidean metric on $\R^n$ and $\C^n$.
For $\vec{x}\in k^n$, define
\[
    \norm{\vec{x}}_{k_{\infty}} := \prod_{i=1}^{l_1} \norm{\sigma_i(\vec{x})} \cdot \prod_{j=1}^{l_2} \norm{\sigma_{l_1+j}(\vec{x})}^2.
\]
For $x\in \bmX(k)$, define
\[
        N_{\Gamma.x, T} := \# \left\{
         y \in \Gamma.x \midd  \norm{y}_{k_{\infty}} \leq T
        \right\}.
\]
Let $K_0:= \prod_{l_1} \SO_2(\R) \times \prod_{l_2} \SU_2(\R)$ be a maximal compact subgroup of $G:= \SL_2(k_{\infty})$.
For a compactly supported smooth $K_0$-invariant function $\varphi : Y:= G /\Gamma \to \R$, define the smoothed count
\[
        N_{\Gamma.x, T} (\varphi )  := 
       \int_Y  \varphi([g])  N_{g\Gamma.x, T}  \, \diff\rmm_{Y}([g]).
\]
\begin{thm}\label{theorem_smoothed_count_quadric}
For every rational point $x\in X(k)$ and $\varphi \in C_c^{\infty}(K_0 \bs Y)$ with $\int_Y \varphi(y) \,\diff\rmm_Y(y) =1 $,
there exist two monic polynomials $p_1,p_2$ of degree $l_1+l_2$ and $l_1+l_2-1$ respectively, $c_1,c_2>0$, $c_2(\varphi) \in \R$ and $\delta\in (0,0.5)$ such that
\[
    N_{\Gamma.x, T} (\varphi) = c_1 T p_1(\log{T})+ c_2(\varphi) c_2 T^{l_1+l_2} p_2(\log{T})  +  O (T^{l_1+l_2-\delta}).
\]
\end{thm}


It should be possible to drop the $K_0$-invariance condition,
however, this would require an analysis on a more complicated compactification, which is not carried out here.
With a different norm, a qualitative version of this was obtained in \cite{zhangrunlinAnnalen2019}.

\begin{rmk}
If one defines 
\[
   B_T:= \left\{ gH\in G/H \midd \norm{g.x}_{k_{\infty}} \leq T \right\}
   ,\quad
   \widetilde{B}_T:= \left\{
    gH^{(1)} \in G/H^{(1)} \midd gH \in B_T,\; \norm{g.e_1}_{k_{\infty}} >1,\;\norm{g.e_1}_{k_{\infty}}>1 \right\},
\]
then the theorem above actually states 
\[
    N_{\Gamma.x, T} (\varphi) = \rmm_{G/H^{(1)}} (\widetilde{B}_T) + c_2(\varphi) \rmm_{G/H}(B_T) + O (T^{l_1+l_2-\delta}).
\]
The constant $c_2(\varphi) = \scrD^+(\varphi) + \scrD^-(\varphi)$ is related to Eisenstein series. For $\Gamma = \SL_2(\calO_k)$, this is explained in App.\ref{appendix_C} and implies the count
 \[
    N_{\Gamma.x, T}  = \rmm_{G/H^{(1)}} (\widetilde{B}_T) + c_2 \rmm_{G/H}(B_T) + O (T^{l_1+l_2-\delta})
\]
for some $c_2\in \R$ being the evaluation of the function $c^{+}+c^{-}$ at the identity coset.
\end{rmk}

\subsection{Equidistribution of infinite homogeneous measures}
By the unfolding and folding arguments from \cite{DukRudSar93, EskMcM93}. the orbital counting statement is related to effective equidistribution.
To state it, we let 
\begin{equation}\label{equation_definition_H}
   H : = \left\{ 
       h_{\vec{t} }:=
       \left(
       \begin{bmatrix}
           t_1  &   0  \\
          0 &    t_1^{-1}
    \end{bmatrix}
    ,...,
     \begin{bmatrix}
           t_{l_1+l_2} &   0  \\
          0 &    t_{l_1+l_2}^{-1}
    \end{bmatrix}
    \right)
     \midd \vec{t} = (t_1,...,t_{l_1+l_2}) \in k_{\infty}^{\times}
   \right\}
\end{equation}
Consider the map
\begin{equation*}
\begin{aligned}
   \mathrm{Nm}: H \cong k_{\infty}^{\times } \cong  ( \R^{\times}  )^{l_1} \oplus (\C^{\times } )^{l_2} &\to \R^+
   \\
   (t_i)_{i=1,...,l_1+l_2} &\mapsto 
   \prod_{i=1}^{l_1} |t_i| \cdot \prod_{j=1}^{l_2} |t_{l_1+j} | ^2
\end{aligned}
\end{equation*}
and let $H^{(1)}:= \mathrm{ker}(\mathrm{Nm})$. On the other hand, there is a subgroup of $H$ parametrized by
\[
   H^{\Delta} : = \left\{ 
       h^{\Delta}_{t}:= h_{(t,...,t)} \midd {t}\in \R^+
   \right\}.
\]
Taking product induces an isomorphism (for any arithmetic subgroup $\Gamma$)
\begin{equation}\label{equation_split_H_mod_Gamma}
   H^{\Delta} \times H^{(1)}\Gamma/\Gamma \cong H\Gamma/\Gamma.
\end{equation}
On the left hand side, there is a natural Haar measure $t^{-1}\diff t \otimes \rmm _{[H^{(1)}]}$ where $ \rmm _{[H^{(1)}]}$ denotes the unique probability Haar measure on $H^{(1)}\Gamma/\Gamma$, a compact subset by Dirichlet's unit theorem.
Given an arithmetic subgroup $\Gamma$, let $\rmm_G, \rmm_H, \rmm_{G/H}$ be the unique Haar measures on $G,H, G/H$ respectively such that 
\begin{itemize}
     \item they are compatible, i.e. $\int_G f(g) \,\diff \rmm_G(g)= \int_{G/H} \int_H f(gh) \,\diff\rmm_H(h)\diff \rmm_{G/H}(gH)$ for any compactly supported continuous function $f$ on $G$;
     \item the induced measure $\rmm_{Y}$ on the quotient $Y:=G/\Gamma$  is a probability measure;
     \item the induced measure $\rmm_{[H]}$ on $H \Gamma/\Gamma$ coincides with $t^{-1}\diff t \otimes \rmm _{[H^{(1)}]}$ via Eq.(\ref{equation_split_H_mod_Gamma}).
\end{itemize}

Last, let $e_1, e_2$ be the standard basis of $k^2$. 

\begin{thm}\label{theorem_equidistribution_infinite_homogeneous_quadric}
There exist $\delta\in (0,1)$ and distributions  $\scrD^+,\scrD^-$ such that the following holds.
For $g \in G$ and $\varphi \in C_c^{\infty} (Y)$ that is $K_0$-invariant,
\[
   \int \varphi(g.z) \, \diff \rmm_{[H]}(z)
   = \frac{1}{l} \log{T_g} \cdot \int \varphi(y) \, \diff \rmm_{Y}(y)
   + \scrD^+(\varphi) + \scrD^-(\varphi) + O_{\supp (\varphi)}(\calS(\varphi) T_g^{-\delta})
\]
where $T_g :=  \norm{g.e_1}_{k_{\infty}} \norm{g.e_2}_{k_{\infty}}$.
\end{thm}

The precise definition of $\scrD^{\pm}(\varphi)$ will be given in Eq.(\ref{equation_distributions_+-}).

\subsection{Outline of the paper}
This work is largely inspired by \cite{OhSha14, KelKon18}. Let us briefly explain the main content of this paper.

Once we fixed some homogeneous measure $\rmm_{[H]}$, its translates form a homogeneous space of the form $G/H$. Its limiting behaviour, then, is encoded into certain compactification of this space.
The main novelty of the paper is the observation that effective equidistribution of \textit{finite} homogeneous measures can be rephrased into Holder continuity of certain functions on this compactification space.
This compactified space would have some corner structure and the effective equidistribution of  \textit{infinite} homogeneous measures is reduced to the effective equidistribution of certain hyperbolas in the corner.
From this perspective, the appearance of the second order term could be explained as some kind of focusing phenomenon as in \cite{EskMozSha96}. 


We deduce the counting from the equidistribution (Sec.\ref{2.2}) first and then prove the effective equidistribution assuming Holder continuity property of certain functions (Sec.\ref{2.3}).
The proof of the latter one is the key. The corresponding effective equidistribution (both the focusing and the generic case) is not available and we have to prove it (Sec.\ref{2.4} and \ref{2.5}) building on \cite{OhSha14}. The proof of Holder continuity (actually a weaker statement will be proved since no metric structure is put on the compactification space here) is then not hard.

Some isolated arguments are collected in the appendix.
In App.\ref{appendix_A}, we prove the volume asymptotic needed for Theorem \ref{theorem_smoothed_count_quadric} using height zeta functions.
The equidistribution of lines in tori is discussed in App.\ref{appendix_B}. Finally in App.\ref{appendix_C}, we indicate how Theorem \ref{theorem_smoothed_count_quadric} might be upgraded to an authentic count using Eisenstein series.

\section{Integer points on ternary affine quadrics}\label{section_affine_quadrics}

\subsection{Preliminaries}

Recall $\bmG=\SL_2$ and $G= \SL_2(k_{\infty})$.
For a subset $E$ of $G$, we let $[E]$ be its image in $Y=G/\Gamma$.
For $s\in \R^{+}$ and $t\in \C^{\times}$, let
\[
     a_{s }:=
       \begin{bmatrix}
        \frac{ s+s^{-1}}{2} &     \frac{ s - s^{-1}}{2}  \\
           \frac{  s -  s^{-1}}{2}  &    \frac{  s +  s^{-1}}{2}
    \end{bmatrix},
    \quad 
    h_{ t }:=
       \begin{bmatrix}
          t &   0 \\
           0 &  t^{-1}
          \end{bmatrix}.
\]
Also,  
\begin{equation*}\label{equation_definition_A}
   A:= 
    \left\{ 
       a_{\vec{s} }:=
        \big( a_{s_1},...,a_{s_{l_1+l_2}}    \big)
     \midd \vec{s} = ( s_1,...,s_{l_1+l_2}) \in (\R^{+})^{l_1+l_2}
   \right\}.
\end{equation*}
For $\vec{t} \in k_{\infty} $, let $h_{\vec{t}}$ be as in Eq.(\ref{equation_definition_H}).
We will also need the stable and unstable horospherical subgroups attached to $h^{\Delta}_t = h_{(t,...,t)}$.
For  $\vec{r}=(r_1,...,r_{l_1+l_2}) \in k_{\infty} $, define
\[
  u_{\vec{r}} := u^+_{\vec{r}} := \left(
   \begin{bmatrix}
          1 & r_1 \\
          0 & 1
      \end{bmatrix},..., 
       \begin{bmatrix}
          1 & r_{l_1+l_2} \\
          0 & 1
      \end{bmatrix} \right),\quad
      u^-_{\vec{r}} := \left(
   \begin{bmatrix}
          1 & 0 \\
          r_1 & 1
      \end{bmatrix},..., 
       \begin{bmatrix}
          1 &  0 \\
          r_{l_1+l_2} & 1
      \end{bmatrix} \right)
\]
and
\[
   U:= U^+ := \left\{ u_{\vec{r}} \midd \vec{r} \in k_{\infty} \right\} ,\quad
   U^- := \big \{ u^-_{\vec{r}} \,\big \vert\, \vec{r} \in k_{\infty} \big \}.
\]

We will take the following form of effective mixing, a special case of \cite[Theorem 10.2]{BenOh12}, as a blackbox.
\begin{thm}\label{theorem_effective_mixing}
There exists some $\delta \in (0,0.5)$ such that for every $g\in G$ and every pair of compactly supported smooth functions $\varphi_1, \varphi_2$ on $Y$, one has
\[
     \normm{
      \int_Y \varphi_1(g .y) \varphi_2(y) \, \diff \rmm_Y(y)
      - \int_Y \varphi_1(y) \, \diff \rmm_Y(y) \int_Y \varphi_2(y) \, \diff \rmm_Y(y)
     }
     \leq \calS(\varphi_1)\calS(\varphi_2)( \norm{g.e_1}_{k_{\infty}}\norm{g.e_2}_{k_{\infty}}   )^{-\delta}.
\]
\end{thm}
In the above the Sobolev norm $ \calS(\varphi) = \sum \norm{\calD(\varphi)}_{L^2}$ for finitely many differential operators. The reader is referred to \cite[Section 3]{GorOhMau08} for more details.


\subsection{From equidistribution to counting} \label{2.2}
In this subsection we prove Theorem \ref{theorem_smoothed_count_quadric} assuming Theorem  \ref{theorem_equidistribution_infinite_homogeneous_quadric}.

Fix some $x\in \bmX(k)$. The stabilizer of $x$ in $\bmG$ is a $k$-split subtorus. Since split sub-tori of $\bmG$ are pairwise $\bmG(k)$-conjugate, there exists some $q_x \in \bmG(k)$ such that $x= q_x.x_0$ where $x_0\in \bmX(k)$ has its stabilizer being equal to $\bmH$. Let $\Gamma^x:= q_x^{-1}\Gamma q_x$ and $Y_x:= G/\Gamma^x$.  Let $\rmm_{Y^x}$ be the $G$-invariant probability measure on $Y^x$. We have a $G$-equivariant isomorphism
\begin{equation*}
\begin{aligned}
      Y^x \cong Y ,\quad g\Gamma^x \mapsto gq_x^{-1}\Gamma.
\end{aligned}
\end{equation*}
Let $\varphi$ be as in the theorem and $\varphi_x$ be the pull-back of $\varphi$ along this isomorphism, which remains $K_0$-invariant.
Therefore,
\begin{equation*}
\begin{aligned}
     N_{\Gamma.x, T}(\varphi) 
     &= \int_{[g]\in Y^x}  \varphi_x([g]) \#
     \left\{
     y\in g\Gamma^x .x_0 \midd \norm{y}  \leq T 
     \right\}
     \, \diff\rmm_{Y^x}([g]).
\end{aligned}
\end{equation*}
By the unfolding and folding argument (see \cite[Section 5]{EskMcM93}), it can be rewritten as
\begin{equation*}
\begin{aligned}
     N_{\Gamma.x, T}(\varphi) 
     &= \int_{ gH \in B_T}
     \int_{H\Gamma^x/\Gamma^x}
     \varphi_x(g.y) \, \diff\rmm_{[H]}(y)
     \diff\rmm_{G/H}([g])
\end{aligned}
\end{equation*}
where $B_T:= \left\{ gH \in G/H  \midd \norm{g.x_0} \leq T \right\}$.
Let us note that $G.x_0 = \bmX(k_{\infty})$.
On applying Theorem \ref{theorem_equidistribution_infinite_homogeneous_quadric} to the right hand side, we find that (recall $T_g :=  \norm{g.e_1}_{k_{\infty}} \norm{g.e_2}_{k_{\infty}}$)
\begin{equation}\label{equation_smooth_count_quadratic}
\begin{aligned}
     &\;N_{\Gamma.x, T}(\varphi) 
     \\
     =&\;  \int_{ gH \in B_T} 
     \frac{1}{l} \log {T_g} \,\diff\rmm_{G/H}(gH)
     + c_2(\varphi) \rmm_{G/H}(B_T)
     +   \bmO \Big(  C \calS(\varphi) \int_{gH \in B_T} T_g^{-\delta}\diff \rmm_{G/H}(gH) \Big)
\end{aligned}
\end{equation}
where $c_2(\varphi)=  \scrD^+ (\varphi)+\scrD^{-}(\varphi)$. 
It remains to compute the volume asymptotics appearing in Eq.(\ref{equation_smooth_count_quadratic}).
We work this out in the appendix using height zeta functions.
In particular, as a consequence of Lemma \ref{lemma_height_zeta_meromorphic_continuation_quadratic} and Tauberian theorem (see e.g.  \cite[Theorem A.1]{Chamber-Loir_Tschinkel_2010_Igusa_integral}), we have for some $\delta \in (0,0.5)$, $c_1,c_2 \neq 0$ and $C_1>1$,
\begin{equation*}
\begin{aligned}
       \frac{1}{l}  \int_{ gH \in B_T} 
     \log {T_g} \,\diff\rmm_{G/H}([g]) &= c_1 T p_1(\log{T})  + O(T^{1-\delta}),\;\;
     \\
      \rmm_{G/H}(B_T)  &= c_2 T p_2(\log{T})  + O(T^{1-\delta})
      \\
      \int_{gH \in B_T} T_g^{-\delta}\diff \rmm_{G/H}(gH)  & \leq C_1 T^{1-\delta}
\end{aligned}
\end{equation*}
where $p_1$ (resp. $p_2$) is a monic polynomial of degree $l_1+l_2$ (resp. $l_1+l_2-1$).
And this is sufficient to conclude the proof of Theorem \ref{theorem_smoothed_count_quadric}.

\subsection{Holder continuity and effective equidistribution}\label{2.3}

Write  $e^{\vec{t}}:=  (e^{t_1},...,e^{t_{l_1+l_2}})$ for $\vec{t} = (t_1,...,t_{l_1+l_2}) \in k_{\infty}$.
For $\ep\in (0,1)$,
\[
   H[\ep]:= \left\{
   h_{e^{\vec{t}}} 
     \midd \normm{t_i} < \ep ,\;\forall \, i=1,...,l_1;\;
    \normm{\mathrm{Re}(t_{ j + l_1})  },  \normm{\mathrm{Im}(t_{  j + l_1})  } < \ep,\; \forall \, j=1,...,l_2
   \right\}.
\]
We choose $\ep_0>0$ small enough such that taking product induces an isomorphism 
\[
H[\ep_0] \times U^{\star}/U^{\star}\cap \Gamma \cong H[\ep_0]U^{\star} \Gamma/ \Gamma ,\quad \star = +,- .
\]
For every $z\in [H^{(1)}]$, we let $ \rmm_z$ be the normalized probability measure supported on $H[\ep_0].z$ defined by the restriction of $H$-invariant Haar measures.
And we let 
\[
\rmm_{\ep_0}:= \int_{  [H^{(1)}] } \rmm_z \, \diff \rmm_{[H^{(1)}]  }(z).
\]
 In particular, $\rmm_{\ep_0}$ is $H^{(1)}$-invariant.
One has a continuous map
\begin{equation*}
\begin{aligned}
     \Psi_{\ep_0}: A \times H^{\Delta} &\to  \Prob(K_0 \bs Y)
    \\
    (a,h) & \mapsto  \Psi_{\ep_0}(a,h):=(ah)_* \rmm_{\ep_0}.
\end{aligned}
\end{equation*}
We will prove effective equidistributions for this family of measures in Proposition \ref{proposition_equidistribution_focusing_affine_quadric},  \ref{proposition_effective_equidistribution_generic_quadric} and encode them into certain continuity property of functions on some compactification of $A\times H^{\Delta}$, which we now describe.
Let $\overline{A}:= A\cup\{\infty\}$ be the one-point compactification of $A$ and consider the embedding
\begin{equation*}
\begin{aligned}
   \iota: A \times H^{\Delta} &\to  \bmP^1(\R) \times \bmP^1(\R) \times \overline{A}
    \\
    (a,h) & \mapsto \iota(a,h):= \left(
    [1:\norm{ah. e_1}_{k_\infty}] , [1:\norm{ah.e_2}_{k_\infty}] , a
    \right).
\end{aligned}
\end{equation*}
Let $\overline{A\times H^{\Delta}}$ be the closure of $\iota(A\times H^{\Delta})$ in $\bmP^1(\R) \times \bmP^1(\R) \times \overline{A}$.
Consider the open subset parametrized by
\begin{equation*}
\begin{aligned}
    \iota_{\infty}:  [0,+\infty) \times [0,+\infty) \times \overline{A} &\to  \bmP^1(\R) \times \bmP^1(\R) \times \overline{A}
    \\
    (x_1,x_2,a) & \mapsto\left(
    [x_1:1 ] , [x_2:1 ] , a
    \right).
\end{aligned}
\end{equation*}
One can check that its image contains $\iota({A\times H^{\Delta}})$ and we 
let $\calN_0$ be the preimage of $\overline{\iota(A\times H^{\Delta})}$ in $ [0,+\infty) \times [0,+\infty) \times \overline{A}$, which is precisely the closure of 
\begin{equation*}
  \calN^{\circ}_0 := \left\{
  (x_1,x_2,a)\in \R^+\times \R^+ \times {A} \midd  x_1^{-1} x_2^{-1} =  \norm{a.e_1}_{k_{\infty}}  \norm{a.e_2}_{k_{\infty}} 
  \right \}.
\end{equation*}
Given $t\in \R^+$ and $a\in A$,
\begin{equation}\label{equation_definition_x(t)}
   x_1:=x_1(t,a):= \frac{1}{
       \norm{ah^{\Delta}_t.e_1}_{k_{\infty}}
   } = t^{-l} \frac{1}{ \norm{a.e_1}_{k_{\infty}}
   },\quad
   x_2:=x_2(t,a):= \frac{1}{
       \norm{  ah^{\Delta}_t.e_2}_{k_{\infty} }
   } = t^{l} \frac{1}{ \norm{a.e_2}_{k_{\infty}}
   }
\end{equation}
satisfies $(x_1,x_2,a)\in \calN_0^{\circ}$.
Conversely, given $ (x_1,x_2,a) \in \calN_0^{\circ}$, there exists a unique $t\in \R^+$ such that $x_1=x_1(t,a)$, $x_2= x_2(t,a)$.
Indeed it is defined by $(a,h^{\Delta}_t)= \iota^{-1}\circ  \iota_{\infty}(x_1,x_2,a)$. 
Given a  $K_0$-invariant function $\varphi\in C_c^{\infty}(Y)$, we let 
\begin{equation}\label{equation_f_quadratic}
      f^z_{\varphi}(x_1, x_2, a) := \int \varphi(ah_t^{\Delta}.y)\, \diff \rmm_z(y),\quad
      f_{\varphi}(x_1,x_2,a):= \int \varphi(ah_t^{\Delta}.y)\, \diff \rmm_{\ep_0}(y)
\end{equation}
the latter of which is nothing but the integration of $\varphi$ against the measure $ \Psi_{\ep_0}\circ \iota^{-1}\circ  \iota_{\infty}(x_1,x_2,a)$.

\begin{thm}\label{theorem_Holder_continuous_quadrics}
There exists $\delta\in(0,0.5)$, for every compact subset $\scrC \subset Y$, there exists $C >1$  such that the following holds.
 For every $K_0$-invariant and smooth  function $\varphi :Y \to \R$ supported on $\scrC$,  the function $f_{\varphi}: \calN^{\circ}_0 \to \R$ extends continuously to $\calN_0$ and satisfies
\begin{equation*}
\begin{aligned}
    \normm{
      f_{\varphi} (x_1,x_2,a) - f_{\varphi} ( 0, x_2, \infty) 
    }
    & \leq C \calS(\varphi ) x_1^{\delta}, \quad \text{ if } x_1 \leq x_2 ;
    \\
     \normm{
      f_{\varphi} (x_1,x_2,a) - f_{\varphi} ( x_1, 0, \infty) 
    }
    & \leq C \calS(\varphi ) x_2^{\delta}, \quad \text{ if } x_2 \leq x_1;
    \\
     \normm{
       f_{\varphi} ( 0, x_2, \infty) - f_{\varphi} ( 0, 0, \infty) 
    }
    & \leq C \calS(\varphi ) x_2^{\delta} ;
    \\
     \normm{
       f_{\varphi} ( x_1, 0, \infty) - f_{\varphi} ( 0, 0, \infty) 
    }
    & \leq C \calS(\varphi ) x_1^{\delta};
\end{aligned}
\end{equation*}
and that $x_2 \mapsto f_{\varphi} ( 0, x_2, \infty)$, $x_1\mapsto f_{\varphi} ( x_1, 0, \infty) $ are differentiable away from $0$.
The boundary value of $f_{\varphi}$ is defined in Eq.(\ref{equation_f_boundary_quadratic}) below.
\end{thm}

\subsubsection{Boundary values of $f_{\varphi}$}

Rather than some abstract continuation, the value of  $ f_{\varphi}  $ at boundary is explicitly determined.

Recall that $\ep_0$ is small enough such that $H[\ep_0] \times U/U\cap \Gamma \cong H[\ep_0]U \Gamma/ \Gamma$.
For each $z \in [H^{(1)}]$, fix a lift $h_{\vec{t}_z}\in H^{(1)}$ of $z$. We may assume that all these lifts are contained in some compact subset. The map
\begin{equation*}
\begin{aligned}
     H[\ep_0] \times U/U\cap \Gamma &\to H[\ep_0]U.z = h_{\vec{t}_z}H[\ep_0]U\Gamma/\Gamma
     \\
     (a,y) & \mapsto h_{\vec{t}_z} a . y
\end{aligned}
\end{equation*}
is a homeomorphism for every $z\in [H^{(1)}]$. 
We normalize $\rmm_{H[\ep_0]}$, the restriction of Haar measures to $H[\ep_0]$ such that for every $z\in [H^{(1)}]$,  $\rmm_{H[\ep_0]}  \otimes \delta_{[\id]}  \cong  \rmm_z$ under this homeomorphism.
Define a $U$-invariant probability measure $ \rmm_z^U$ supported on $H[\ep_0]U.z $ such that
\[
   \rmm_{H[\ep_0]} \otimes \rmm_{[\rmU]} \cong \rmm_z^U
\]
where $\rmm_{[\rmU]}$ is the $U$-invariant probability measure supported on $U\Gamma/\Gamma$.
Another probability measure $\rmm_z^{U^-}$ supported on $ H[\ep_0]U^-.z$ is similarly defined. 
Let
\[
   \rmm_{\ep_0}^U := \int_{ [  H^{  (1)  } ]  } \rmm_z^U \, \diff \rmm_{[H^{(1)} ]}(z),   \quad
    \rmm_{\ep_0}^{U^-} := \int_{ [  H^{  (1)  } ]  } \rmm_z^{U^-} \, \diff \rmm_{[H^{(1)} ]}(z).
\]
Now,
\begin{equation}\label{equation_f_boundary_quadratic}
\begin{aligned}
    \begin{cases}
    \displaystyle
      f_{\varphi}(x_1, 0 ,\infty  ) = \int \varphi \left( h^{\Delta}_{ x_1^{-1/l} } . y \right) \,\diff\rmm_{\ep_0}^U(y) \\
      \displaystyle
      f_{\varphi}( 0 , x_2 ,\infty  )   = \int \varphi  \left( h^{\Delta}_{ x_2^{1/l} } . y  \right) \,\diff\rmm_{\ep_0}^{U^-}(y) \\
      \displaystyle
      f_{\varphi}( 0 , 0,\infty  )   = \int \varphi(y) \, \diff \rmm_{Y}(y)
    \end{cases}.
\end{aligned}
\end{equation}

Theorem \ref{theorem_Holder_continuous_quadrics} is a combination of Proposition \ref{proposition_equidistribution_focusing_affine_quadric}  and \ref{proposition_effective_equidistribution_generic_quadric}, to be discussed later. 
Before that, we explain how to deduce Theorem \ref{theorem_equidistribution_infinite_homogeneous_quadric} in the next subsection.

\subsubsection{Proof of Theorem \ref{theorem_equidistribution_infinite_homogeneous_quadric} assuming  Theorem \ref{theorem_Holder_continuous_quadrics}}
Thanks to the ``$KAH$-decomposition'' (see \cite[Proposition 7.1.3]{Schlichtkrull_symmetric_spaces}),
\[
      K_0 \times A \times H \to G
\]
is surjective. So it suffices to prove  Theorem \ref{theorem_Holder_continuous_quadrics} when $g=a \in A$. 
Fix some $K_0$-invariant $\varphi \in C_c^{\infty}(Y)$ and $a\in A$.
Recall that 
\begin{equation*}
\begin{aligned}
         H \Gamma/ \Gamma &\cong H^{\Delta} \times  H^{(1)} \Gamma / \Gamma 
         \\
         \rmm_{[H]} & \cong \frac{\diff t}{t} \otimes \rmm _{[H^{(1)}]}.
\end{aligned}
\end{equation*}
So (see Eq.(\ref{equation_definition_x(t)}) for the definition of $x_i(t,a )$)
\begin{equation*}
\begin{aligned}
    \int \varphi(a.z) \,\diff\rmm_{[H]}(z) &= \int \int \varphi (ah^{\Delta}_t.z) \, \frac{\diff t}{t}  \, \diff \rmm _{[H^{(1)}]}(z)
    \\
    &=
    \int_z \int_y \int_{t\in \R^+}  \varphi(a h^{\Delta}_t.y) \, \frac{\diff t}{t} \diff\rmm_{z}(y) \diff \rmm _{[H^{(1)}]}(z)
    \\
    &=
    \int_{t\in \R^+}
     f_{\varphi} (x_1(t,a),x_2(t,a),a) \, \frac{\diff t}{t}.
\end{aligned}
\end{equation*}
Break $ \int_{\R^+}$ into $\int_{1}^{\infty} + \int_{0}^1$
and note that
\[
   \frac{\diff x_2}{x_2} = \frac{\diff x_1}{x_1} = l \cdot  \frac{\diff t}{t}, \quad
   x_1(t,a) > x_2(t,a) \iff 0 <t <1.
\]
Let $C$ and $\delta$ be as in Theorem \ref{theorem_Holder_continuous_quadrics}.
 We have, 
\begin{equation*}
\begin{aligned}
     & \int_{t\in \R^+}
     f_{\varphi}(x_1(t,a),x_2(t,a),a) \, \frac{\diff t}{t}
     \\
     =&\,
     \int_1^{\infty} f_{\varphi}(0,x_2(t,a),\infty) + \bmO( C\calS(\varphi)x_1(t,a)^{\delta}) \,\frac{\diff t }{t}
     +
       \int_0^{1} f_{\varphi}(x_1(t,a),0, \infty) +  \bmO( C\calS(\varphi)x_2(t,a)^{\delta}) \,\frac{\diff t }{t}
       \\
       =&\,
       \frac{1}{l}
       \int_{   \norm{a.e_2}_{k_{\infty}}^{-1}
       }^{\infty}  f_{\varphi}(0,x_2,\infty) \,\frac{\diff x_2}{x_2}
       + \frac{1}{l}
       \int_{   \norm{a.e_1}_{k_{\infty}}^{-1}
       }^{\infty}  f_{\varphi}(x_1,0,\infty) \,\frac{\diff x_1}{x_1}
       + \bmO \Big(   \frac{2C\calS(\varphi)}{\delta l}
        \norm{a.e_1}_{k_{\infty}}^{-\delta}
       \Big)
\end{aligned}
\end{equation*} 
where we used $  \norm{a.e_1}_{k_{\infty}} = \norm{a.e_2}_{k_{\infty}} $.
For the first summand, we may continue as 
\begin{equation*}	
\begin{aligned}
    &
    \frac{1}{l} \int_{   \norm{a.e_2}_{k_{\infty}}^{-1}
    }^1 \frac{  f_{\varphi}(0,x_2,\infty)- f_{\varphi} (0,0,\infty)
    }{ x_2
    } \,\diff x_2
    + \frac{1}{l}
    f_{\varphi}(0,0,\infty) \int_{ \norm{a.e_2}_{k_{\infty}}^{-1}
    }^1 \frac{\diff x_2}{x_2}
    + \frac{1}{l} \int_1^{\infty} f_{\varphi}(0,x_2,\infty)\,\frac{\diff x_2}{x_2} 
    \\
    = \,&
    \frac{1}{l} \int_0^1 \frac{
     f_{\varphi}(0,x_2,\infty)- f_{\varphi} (0,0,\infty)
    }{  x_2 } \diff x_2
    +    \frac{1}{l} \log\left(
      \norm{a.e_2}_{k_{\infty}}
    \right) f_{\varphi}(0,0,\infty) 
     +
    \frac{1}{l} \int_1^{\infty} \frac{
     f_{\varphi}(0,x_2,\infty)
    }{  x_2 } \diff x_2 
    \\
    \, &
    + \bmO \Big(   \frac{C\calS(\varphi)}{\delta l} \norm{a.e_2}_{k_{\infty}}^{-\delta} \Big)
    \\
    =\,&
     \frac{1}{l} \log\left(
      \norm{a.e_2}_{k_{\infty}}
    \right) f_{\varphi}(0,0,\infty) +
    \frac{1}{l}   \int_0^{\infty} \frac{\partial  f_{\varphi}}{\partial x_2} 
    (0,x_2,\infty) (- \log(x_2)) \diff x_2  
      +  \bmO \Big(   \frac{C\calS(\varphi)}{\delta l} \norm{a.e_2}_{k_{\infty}}^{-\delta} \Big)
\end{aligned}
\end{equation*}
where in the last step we applied integration by part.
One has similar results for the second summand:
\begin{equation*}
\begin{aligned}
      &
      \frac{1}{l}
       \int_{   \norm{a.e_1}_{k_{\infty}}^{-1}
       }^{\infty}  f_{\varphi}(x_1,0,\infty) \,\frac{\diff x_1}{x_1}
       \\
       = \, &
     \frac{1}{l} \log\left(
      \norm{a.e_1}_{k_{\infty}}
    \right) f_{\varphi}(0,0,\infty) +
    \frac{1}{l}   \int_0^{\infty} \frac{\partial   f_{\varphi}}{\partial x_1} 
   (x_1,0,\infty) (- \log(x_1)) \diff x_1 
      +  \bmO \Big(   \frac{C\calS(\varphi)}{\delta l} \norm{a.e_1}_{k_{\infty}}^{-\delta} \Big).
\end{aligned}
\end{equation*}

 Combining efforts so far and recalling Eq.(\ref{equation_f_boundary_quadratic}), we obtain
\begin{equation}\label{equation_orbital_integral_full}
\begin{aligned}
    & \int \varphi(a.z) \,\diff\rmm_{[H]}(z)
    \\
      = &\,
     \frac{1}{l} \log\left(
     \norm{a.e_1}_{k_{\infty}}
      \norm{a.e_2}_{k_{\infty}}
    \right) \int \varphi(y) \, \diff \rmm_{Y}(y)
    \\
     &
     + \frac{1}{l}  \int_{0}^{\infty} 
     \frac{\diff}{\diff x} \Bigg\vert_{x=x_1}
      \left(  
     \int \varphi \Big( 
      h^{\Delta}_{ {x^{-1/l}}  } .z 
     \Big) \diff \rmm_{ \ep_0 }^U (z)
     \right)
     (-\log x_1) \, \diff x_1
    \\
    & +
     \frac{1}{l} \int_{0}^{\infty} \frac{\diff}{\diff x} \Bigg \vert_{x=x_2}
     \left(  \int \varphi \Big( 
      h^{\Delta}_{  {x^{1/l}}   } .z 
    \Big) \diff \rmm_{ \ep_0 }^{U^-} (z)  
    \right)
    (-\log x_2) \, \diff x_2
    +\bmO \Big(   \frac{4C\calS(\varphi)}{\delta l}  \norm{a.e_1}_{k_{\infty}}^{-\delta} \Big).
\end{aligned}
\end{equation}
The dependence on $\ep_0$ can be eliminated as follows.
By assumption, there exists a probability measure $\nu$ (depending on $\ep_0$) on $\R^+$ such that 
\[
      \rmm_{\ep_0}  = 
       \int (h^{\Delta}_{\theta})_* \rmm_{[H^{(1)}]}  \,\diff\nu(\theta),
\]
which implies that
\[
    \rmm_{ \ep_0 }^{U^{\star}}   = 
       \int (h^{\Delta}_{\theta})_* \rmm^{U^{\star}}_{[H^{(1)}]}  \,\diff\nu(\theta),\; \text{ for } \star =+,-,
\]
where $ \rmm^{U^{\star}}_{[H^{(1)}]} $ denotes the invariant probability measure supported on $H^{(1)} U^{\star}\Gamma/\Gamma  $.
On the other hand for each fixed $\theta \in \R^+$, upon change of variable $x_2 \to x_2 '=x_2 \theta^l$, 
\begin{equation*}
\begin{aligned}
       &
     \frac{1}{l} \int_{0}^{\infty} \frac{\diff}{\diff x} \Bigg \vert_{x=x_2}
     \left(  \int \varphi \Big( 
      h^{\Delta}_{  {x^{1/l}}   }h^{\Delta}_{\theta} .z 
    \Big) \diff \rmm_{[H^{(1)}] }^{U^-} (z)  
    \right)
    (-\log x_2) \, \diff x_2
    \\
    = & \,
     \frac{1}{l} \int_{0}^{\infty} \frac{\diff}{\diff x} \Bigg \vert_{x=x'_2}
     \left(  \int \varphi \Big( 
      h^{\Delta}_{  {x^{1/l}}   } .z 
    \Big) \diff \rmm_{[H^{(1)}] }^{U^-} (z)  
    \right)
    (-\log x_2' + \log \theta^l)  \, \diff x'_2
    \\
    =&\,
     \frac{1}{l} \int_{0}^{\infty} \frac{\diff}{\diff x} \Bigg \vert_{x=x'_2}
     \left(  \int \varphi \Big( 
      h^{\Delta}_{  {x^{1/l}}   } .z 
    \Big) \diff \rmm_{[H^{(1)}] }^{U^-} (z)  
    \right)
    (-\log x_2' )  \, \diff x'_2
    -
    \log \theta \cdot \int \varphi(y) \, \diff\rmm_Y(y).
\end{aligned}
\end{equation*}
Similarly, via $x_1 \to x_1 '=x_1 \theta^{-l}$,
\begin{equation*}
\begin{aligned}
       &
     \frac{1}{l} \int_{0}^{\infty} \frac{\diff}{\diff x} \Bigg \vert_{x=x_1}
     \left(  \int \varphi \Big( 
      h^{\Delta}_{  {x^{-1/l}}   }h^{\Delta}_{\theta} .z 
    \Big) \diff \rmm_{[H^{(1)}] }^{U^+} (z)  
    \right)
    (-\log x_1) \, \diff x_1
    \\
    =&\,
     \frac{1}{l} \int_{0}^{\infty} \frac{\diff}{\diff x} \Bigg \vert_{x=x'_1}
     \left(  \int \varphi \Big( 
      h^{\Delta}_{  {x^{-1/l}}   } .z 
    \Big) \diff \rmm_{[H^{(1)}] }^{U^+} (z)  
    \right)
    (-\log x_1' )  \, \diff x'_1
    +
    \log \theta \cdot \int \varphi(y) \, \diff\rmm_Y(y).
\end{aligned}
\end{equation*}
When they add together, the second terms get cancelled.
Hence if we define
\begin{equation}\label{equation_distributions_+-}
\begin{aligned}
     \scrD^+ (\varphi)  &:=\frac{1}{l}  \int_{0}^{\infty} 
     \frac{\diff}{\diff x} \Bigg\vert_{x=x'_1}
      \left(  
     \int \varphi \Big( 
      h^{\Delta}_{ {x^{-1/l}}  } .z 
     \Big) \diff \rmm_{[H^{(1)}] }^U (z)
     \right)
     (-\log x'_1) \, \diff x'_1,
    \\
     \scrD^-(\varphi) &:=
       \frac{1}{l} \int_{0}^{\infty} \frac{\diff}{\diff x} \Bigg \vert_{x=x'_2}
     \left(  \int \varphi \Big( 
      h^{\Delta}_{  {x^{1/l}}   } .z 
    \Big) \diff \rmm_{[H^{(1)}] }^{U^-} (z)  
    \right)
    (-\log x_2') \, \diff x_2',
\end{aligned}
\end{equation}
then Eq.(\ref{equation_orbital_integral_full}) can be written as
\begin{equation*}
\begin{aligned}
      \int \varphi(a.z) \,\diff\rmm_{[H]}(z)
     = \, 
     \frac{1}{l} \log\left(
    T_g
    \right) \int \varphi(y) \, \diff \rmm_{Y}(y)
     +   \scrD^+(\varphi) + \scrD^-(\varphi)
    +\bmO \Big(   \frac{4C\calS(\varphi)}{\delta l} T_g^{-0.5\delta} \Big).
\end{aligned}
\end{equation*}
The proof is now complete.

\subsection{Effective equidistribution: the focusing case}\label{2.4}

Our goal in this subsection is to show
\begin{prop}\label{proposition_equidistribution_focusing_affine_quadric}
There exist $\delta,\, \delta' \in (0,0.5)$ such that for every compact subset $\scrC$ of $Y$, there exists $C>1$ depending on $\scrC$ such that for every $\varphi \in C_c^{\infty}(Y)$ supported on $\scrC$, $z\in [H^{(1)}]$ and $(x_1,x_2,a)\in \calN_0^{\circ}$,
\begin{equation}
\begin{aligned}
      \normm{  f_{\varphi} (x_1, x_2 , a ) - f_{\varphi} (x_1,0, \infty) } &\leq C \Lip( \varphi ) x_2^{\delta'} ,\quad
      \text{ if } x_1 > x_2^{\delta};
      \\
      \normm{  f_{\varphi} (x_1, x_2 , a ) - f_{\varphi} (0 ,x_2, \infty) } &\leq C \Lip( \varphi ) x_1^{\delta'} ,\quad
      \text{ if } x_2 > x_1^{\delta}. 
\end{aligned}
\end{equation}
\end{prop}

We will first work out a special case (Lemma \ref{lemma_focusing_1}) by using a local coordinate system to convert the integrals into averages over lines in a torus, and applying effective equidistribution for such lines (Lemma \ref{lemma_equidistribution_coordinate_lines_tori_number_fields}). Then the general case follows by ``KHU'' decomposition (Lemma \ref{lemma_KHU_decomposition}).

\subsubsection{Convention}

For two vectors $\vec{v}=(v_1,...,v_k)$ and $\vec{w}=(w_1,....,w_k)$, we adopt the convention that 
\[
 \vec{v}.\vec{w}:= (v_1w_1,...,v_kw_k),\;
 e^{\vec{v}} := (e^{v_1} ,..., e^{v_k}),\; \diff \vec{v}:= \diff v_1 ...\diff v_k.
\]
And for $\vec{t} \in k_{\infty}^{\times}$, we let $\vec{t}^{-2}:= (t_1^{-2},...,t_{l_1+l_2}^{-2})$.

The following lemma can be viewed as a special case of the above proposition.
\begin{lem}\label{lemma_focusing_1}
  There exist $C>1$ and $\delta \in(0,0.5)$ such that for every $\varphi\in C_c^{\infty}(Y)$, $\vec{r}\in k_{\infty}$ and $z\in [H^{(1)}]$, we have
  \[
     \left\vert    \int \varphi(u_{\vec{r} }.y) \, \diff\rmm_{z}(y) -  \int \varphi (y) \,\diff \rmm_z^U(y)
     \right\vert
     \leq C \Lip(\varphi) \norm{ u_{\vec{r} } .e_2 } ^{-\delta}.
  \]
  Similarly,
  \[
     \left\vert    \int \varphi(u^-_{\vec{r} }.y) \, \diff\rmm_{z}(y) -  \int \varphi (y) \,\diff \rmm_z^{U^-}(y)
     \right\vert
     \leq C \Lip(\varphi) \norm{ u^{-}_{\vec{r} } .e_1 } ^{-\delta}.
  \]
\end{lem}
To prove this lemma we will work with a model space and, to avoid redundancy, only the first part will be proved. 

\subsubsection{Model space}


Recall that
\begin{equation}
\begin{aligned}
     H[\ep_0] \times U/U\cap \Gamma &\to H[\ep_0]U.z = h_{\vec{t}_z}H[\ep_0]U\Gamma/\Gamma
     \\
     (a,y) & \mapsto h_{\vec{t}_z} a . y
\end{aligned}
\end{equation}
is a homeomorphism for every $z\in [H^{(1)}]$. And $\rmm_{H[\ep_0]} \otimes \rmm_{[\rmU]} \cong \rmm_z^U$ via this homeomorphism.

Let $\Lambda $ be the preimage of $U\cap \Gamma $ under the homeomorphism $\R^l \cong k_{\infty} \cong U$ given by $\vec{r}\mapsto u_{\vec{r}}$, then 
\begin{equation*}
\begin{aligned}
        [-\ep_0 , \ep_0]^l \times \R^l/ \Lambda &\cong H[\ep_0] \times U/U\cap \Gamma
        \\
        (\vec{t} , \vec{r} + \Lambda)  &\mapsto \left( a_{e^{\vec{t} } }, [ u_{\vec{r}} ]  \right)
\end{aligned}
\end{equation*}
and $\Lambda$ is commensurable with (the geometric embedding of) $\calO_k$.
Given $\varphi$, by restricting it to $H[\ep_0]U.z$, we get a smooth function $\psi:  [-\ep_0,\ep_0]^l \times \R^l/ \Lambda  \to \R$ such that
\begin{equation}\label{equation_model_focusing}
\begin{aligned}
         \int \varphi(u_{\vec{r} }.y) \, \diff\rmm_{z}(y)
        & = \frac{1}{(2\ep_0)^l} \int_{[-\ep_0,\ep_0]^l} \psi ( \vec{t}, [e^{-2\vec{t}}. \vec{r}_z] ) \,\diff \vec{t};
         \\
           \int \varphi (y) \,\diff \rmm_z^U(y) 
           &=  \frac{1}{(2\ep_0)^l} \int_{[-\ep_0,\ep_0]^l \times \R^l/\Lambda} \psi ( \vec{t}, [\vec{v}]) \,\diff \vec{t} \, \diff[\vec{v}]
\end{aligned}
\end{equation}
where $\vec{r}_z:=( \vec{t}_z)^{-2}. \vec{r}$ and $\diff[\vec{v}]$ is normalized to be a probability measure on $\R^l/\Lambda$.
Moreover, $C^{-1}\Lip(\varphi) \leq \Lip(\psi) \leq C\Lip(\varphi) $ for some constant dependent only on $\ep_0, \Lambda$ (i.e. only on $\ep_0,\Gamma$).

It suffices to show that the quantities on the right hand side of Eq.(\ref{equation_model_focusing}) are close to each other.
Without loss of generality we are going to assume
\[
   \normm{r_{z,l_1+l_2}} = \max \left\{ \normm{r_{z,i}} ,\; i=1,...,l_1+l_2 \right\}
\]
and this number is at least $10$ for otherwise the conclusion would be direct.

\subsubsection{Effective equidistribution of lines on tori defined by number fields}
In this subsection we prove effective equidistribution of coordinate lines on tori defined by quotients of geometric embedding of $\calO_k$ (see Lemma \ref{lemma_equidistribution_coordinate_lines_tori_number_fields} below). This follows from effective equidistribution of 
``irrational'' lines on tori and Diophantine properties of algebraic numbers. 
The results should be well-known and a proof in the appendix is provided in the appendix.
 Recall $\vec{e}_1,...,\vec{e}_{l_1+l_2}$ are defined via $\R^{l_1} \oplus \C^{l_2}$. So 
\[
   \left(    \vec{e_1},...  , \vec{e}_{l_1}, \vec{e}_{l_1+1}, i \vec{e}_{l_1+1},....
         , \vec{e}_{l_1+l_2}, i\vec{e}_{l_1+l_2} 
  \right)
\]
forms a basis of $\R^{l_1} \oplus \C^{l_2} \cong \R^l$ as a linear space over $\R$. 
\begin{lem}\label{lemma_equidistribution_coordinate_lines_tori_number_fields}
Let $\Lambda$ be a lattice in $\R^l$ commensurable with the geometric embedding of $\calO_k$. 
Take the normalization $\int_{\R^l/\Lambda} \diff[\vec{v}] =1$.
There exists some $0<\delta <0.5$ such that for  every $T_0\in \R$ and $T$ sufficiently large, the following holds for every Lipschitz-continuous function $f$:
\begin{itemize}
    \item[(1)] For $i=1,...,l_1$, 
    \begin{equation*}
    \begin{aligned}
                  \frac{1}{T} \int_{T_0}^{T_0+T} f([t \vec{e}_i  ]) \, \diff t = 
                  \int_{\R^l/\Lambda} f([\vec{v}]) \, \diff [\vec{v}] + \bmO \Big(\frac{\Lip(f)}{T^{\delta}} \Big).
   \end{aligned}
    \end{equation*}
    \item[(2)] For $j=1,..., l_2$ and every interval $I$ of length at most $2\pi$,
     \begin{equation*}
    \begin{aligned}
                  \frac{1}{T} \int_{T_0}^{T_0+T}  \int_{I}
                  f([t e^{i\theta}  \vec{e}_{l_1+j}  ]) \,\diff \theta \diff t = 
                  |I| \int_{\R^l/\Lambda} f([\vec{v}]) \, \diff [\vec{v}] + \bmO \Big(\frac{\Lip(f)}{T^{\delta}} \Big).
   \end{aligned}
    \end{equation*}
\end{itemize}
\end{lem}

\subsubsection{Proof of Lemma \ref{lemma_focusing_1}}

The implicit constants in the big $O(-)$ here will only be dependent on $\ep_0$ and $\Gamma$.

We start with the first equation in Eq.(\ref{equation_model_focusing}).
To simply notations, define
\begin{itemize}
    \item  $ r_{z,l_1+l_2} = T e^{i\theta} \text{ for some }T >0 ,\; \theta \in [0,2\pi)$;
    \item  $\vec{e}_{l_1+l_2}:=(0,...,0,1)\in k_{\infty}$ and for $\vec{v} =(v_1,...,v_{l_1+l_2})\in k_{\infty}$, let $\vec{v}^{(l_1+l_2)} = \vec{v} - v_{l_1+l_2} \vec{e}_{l_1+l_2}$;
    \item $\vec{v}_0:=  \big(e^{-2\vec{t}}. \vec{r}_z  \big)^{(l_1+l_2)}$ and $\vec{t}_0 := \vec{t}^{(l_1+l_2)}$;
    \item $t_{l_1+l_2}^R := \mathrm{Re}(t_{l_1+l_2})$, $s_{l_1+l_2}:= e^{-2t^R_{l_1+l_2}}$ and $t_{l_1+l_2}^I := \mathrm{Im}(t_{l_1+l_2})$.
\end{itemize}
With these notations,
\begin{equation*}
\begin{aligned}
      ( \vec{t},  [ e^{-2\vec{t}} .\vec{r}_z ] )
      &= \left(
         \vec{t}_0+ t_{l_1+l_2} \vec{e}_{l_1+l_2} , [\vec{v}_0 + e^{-2t_{l_1+l_2} }  T e^{i\theta} \vec{e}_{l_1+l_2}
         ]
      \right)
      \\
      &=
      \left(
         \vec{t}_0+ \big(  \frac{\log{s_{l_1+l_2}}}{-2} +i t^I_{l_1+l_2} \big) \vec{e}_{l_1+l_2} , [\vec{v}_0 + e^{-2t_{l_1+l_2} }  T e^{i\theta} \vec{e}_{l_1+l_2}
         ]
      \right).
\end{aligned}
\end{equation*}
Let $M$ be a positive integer whose exact value will be chosen later. 
Define for $m=0,...,M-1$,
\[
   s_m:= \frac{m}{M} (e^{2\ep_0}- e^{-2\ep_0}) + e^{-2\ep_0},\quad t_m^I:= \frac{m}{M}2\ep_0 -\ep_0.
\]
Then
\begin{equation*}
\begin{aligned}
    ( e^{-2\ep_0}, e^{2\ep_0} ] = \bigsqcup_{m=0}^{M-1} (0, \frac{e^{2\ep_0}- e^{-2\ep_0} }{M} ] + s_k,\quad
    (-\ep_0, \ep_0] = \bigsqcup_{m=0}^{M-1} (0, \frac{2\ep_0}{M} ] + t_m^I.
\end{aligned}
\end{equation*}
Note that for every $m_1,m_2 \in \{0,...,M-1\}$ and every $s\in [0, \frac{e^{2\ep_0}- e^{-2\ep_0} }{M} ],\, t^I \in [0, \frac{2\ep_0}{M} ] ,\, [v] \in \R^l/\Lambda$, we have
\begin{equation}\label{equation_approximate_first_coordinate}
\begin{aligned}
    & \, \psi \Big(
    \vec{t_0}+ \big( \frac{\log(s_{m_1}+s)}{-2}  + i (t^I_{m_2}  + t^I )
    \big)  \vec{e}_{l_1+l_2} ,  [v]
    \Big) 
    \\
    = &    \,
     \psi\Big(
    \vec{t_0}+ \big( \frac{\log(s_{m_1})}{-2}  + i t^I_{m_2}  
    \big)  \vec{e}_{l_1+l_2} ,  [v]
    \Big) 
    +O\Big( \frac{\Lip(\psi)}{M} \Big) .
\end{aligned}
\end{equation}
Set furthermore
\[
    \vec{w}_{m_1,m_2}:= \big( \frac{\log(s_{m_1})}{-2}  + i t^I_{m_2}  
    \big)  \vec{e}_{l_1+l_2} ,\quad
     \theta_{m_2}:= \theta -2t_{m_2}^I.
\]
We estimate the right hand side of Eq.(\ref{equation_model_focusing}) by first integrating against 
\[
\diff t_{l_1+l_2} = \diff t^R_{l_1+l_2} \diff t^I_{l_1+l_2} = \frac{ 1}{2 s_{l_1+l_2}} \diff s_{l_1+l_2} \diff t^I_{l_1+l_2}.
\]
We will write $s_{l_1+l_2} = s_{m_1} + s $ and $t^I_{l_1+l_2} = t^I_{m_2}+ t^I $ below.
Also note that $\frac{1}{2(s_{m_1}+s)} = \frac{1}{2s_{m_1}} + O(\frac{1}{M})$ for $s\in (0, \frac{e^{2\ep_0} -e^{-2\ep_0} }{M})$.
We have
\begin{equation}\label{equation_approximation_by_pieces}
\begin{aligned}
       &       \frac{1}{(2\ep_0)^2} \int_{[-\ep_0,\ep_0]^2} \psi ( \vec{t}, e^{-2\vec{t}} r_z+ [w_o] ) \,\diff t_{l_1+l_2}
       \\
       =&
       \frac{1}{(2\ep_0)^2}
       \sum_{m_1,m_2=0}^{M-1}  \int_{0}^{
       \frac{e^{2\ep_0} -e^{-2\ep_0} }{M}
       } 
       \frac{\diff s}{2s_{m_1}}
       \int_0^{\frac{2\ep_0}{M}}
       \diff t^I
       \psi\left(
        \vec{t}_0 + \vec{w}_{m_1,m_2},
        \vec{v}_0 + \big( T(s+ s_{m_1}) e^{i(-2t^I + \theta_{m_2} )}  \big) \vec{e}_{l_1+l_2}
       \right)
       \\
       & + O \Big( \frac{\Lip(\psi)}{M} \Big)
       \\
       =&
       \frac{1}{(2\ep_0)^2 T}
       \sum_{m_1,m_2=0}^{M-1}  \int_{ T s_{m_1} }^{
       T( s_{m_1}+
       \frac{e^{2\ep_0} -e^{-2\ep_0} }{M})
       } 
       \frac{\diff s}{2s_{m_1}}
       \int_0^{ \frac{2\ep_0}{M}}
       \diff t^I
       \psi\left(
        \vec{t}_0 + \vec{w}_{m_1,m_2},
        \vec{v}_0 + \big(s e^{i(-2t^I + \theta_{m_2} )}  \big) \vec{e}_{l_1+l_2}
       \right)
       \\
       &+ O \Big( \frac{\Lip(\psi)}{M}   \Big)
\end{aligned}
\end{equation}

For $m_1,m_2$ fixed, apply part (2) of Lemma \ref{lemma_equidistribution_coordinate_lines_tori_number_fields} to
\[
     f([\vec{v}]) :=  \psi\left(
        \vec{t}_0 + \vec{w}_{m_1,m_2},
        [\vec{v}_0 +  \vec{v}]
       \right),
\]
\[
      j=l_2,\quad
      I= -2 [0 ,\frac{2\ep_0}{M}] + \theta_{m_2},\quad
      T_0= T s_{m_1},\quad
      T= T\frac{e^{2\ep_0} -e^{-2\ep_0} }{M}.
\]
We denote the $\delta$ in Lemma \ref{lemma_equidistribution_coordinate_lines_tori_number_fields} as $\delta_1$ and 
fix a constant $L_3>1$, dependent only on $\Lambda$, such that $L_3^{-1} \Lip(\psi) \leq \Lip(f) \leq L_3 \Lip(\psi)$.
For $T$ sufficiently large,
\begin{equation}\label{equation_apply_equidistribution_line}
\begin{aligned}
         &
          \frac{1}{T}
         \int_{ T s_{m_1} }^{
       T( s_{m_1}+
       \frac{e^{2\ep_0} -e^{-2\ep_0} }{M})
       } 
       \frac{\diff s}{2s_{m_1}}
       \int_0^{ \frac{2\ep_0}{M}}
       \diff t^I
       \psi\left(
        \vec{t}_0 + \vec{w}_{m_1,m_2},
        [\vec{v}_0 + \big(s e^{i(-2t^I + \theta_{m_2} )}  \big) \vec{e}_{l_1+l_2}  ]
       \right)
       \\
       \, &
        \frac{2\ep_0(e^{2\ep_0} -e^{-2\ep_0}) }{M^2 }
        \int_{\R^l/\Lambda}   \psi\left(
        \vec{t}_0 + \vec{w}_{m_1,m_2},
        [ \vec{v}]
       \right) \, \diff[\vec{v}]
       +  {(e^{2\ep_0} -e^{-2\ep_0} )^{1-\delta_1 } L_3}
       \bmO \Big(\frac{\Lip(\psi)}{  T^{\delta_1}M^{-\delta_1+1}} \Big).
\end{aligned}
\end{equation}
By reversing the discussion above Eq.(\ref{equation_approximate_first_coordinate}), we find
\begin{equation}\label{equation_final_approximate}
\begin{aligned}
        &
      \frac{e^{2\ep_0} - e^{-2\ep_0}}{2\ep_0 M^2} \sum_{m_1,m_1=0}^{M-1}  \frac{1}{2m_1} \int_{\R^l/\Lambda}   \psi\left(
        \vec{t}_0 + \vec{w}_{m_1,m_2},
        [ \vec{v}]
       \right)  \diff[\vec{v}]
       \\
       = \,&
       \frac{1}{(2\ep_0)^2}
        \sum_{m_1,m_1=0}^{M-1} 
        \int_{0}^{\frac{  e^{2\ep_0} - e^{-2\ep_0} }{M}
        } 
        \int_{0}^{\frac{2\ep_0}{M}} \int_{\R^l/\Lambda}  \psi\left(
          \vec{t}^{(l_1+l_2)} + \big( \frac{ \log(s_{m_1}+s)  }{-2} + i (t_{m_2}^I + t^I) \big) \vec{e}_{l_1+l_2}
        \right) 
        \\
        &\; \diff [\vec{v}]   \diff t^I  \frac{\diff s}{2(s_{m_1}+s)}+ O(\frac{\Lip(\psi) }{M})
        \\
        =\,&
       \frac{1}{(2\ep_0)^2} \int_{-\ep_0}^{\ep_0} \int_{-\ep_0}^{\ep_0} \int_{\R^l/\Lambda}
       \psi(\vec{t}, [\vec{v}]) \,\diff [\vec{v}] \diff t^R_{l_1+l_2} \diff t^I_{l_1+l_2} + 
       O\Big(   \frac{\Lip(\psi)  }{ M } \Big).
\end{aligned}
\end{equation}
By combining Eq.(\ref{equation_approximation_by_pieces}
, \ref{equation_apply_equidistribution_line}
, \ref{equation_final_approximate}) and Eq.(\ref{equation_model_focusing}), and integrating the rest $\diff t_1 ...\diff t_{l_1+l_2-1}$,
 we find that
\begin{equation*}
\begin{aligned}
      \int \varphi(u_{\vec{r} }.y) \, \diff\rmm_{z}(y) =    \int \varphi (y) \,\diff \rmm_z^U(y) 
      + O \Big(  \Lip(\varphi) (\frac{1}{M} + \frac{M^2}{T^{\delta_1}  M^{-\delta_1+1}  })
      \Big).
\end{aligned}
\end{equation*}
Set $M:= \lfloor T^{0.11\delta_1} \rfloor$. We have 
\[
\frac{1}{M} + \frac{M^2}{T^{\delta_1}  M^{-\delta_1+1}  }
\leq 
T^{-0.1\delta_1} + T^{-\delta_1} T^{2(0.11\delta_1)} \leq 2 T^{-0.1 \delta_1}.
\]
As $T = \max_{i=1,...,l_1+l_2} \left\{ \normm{r_{z,i}} \right\} $ satisfies $C^{-1}T \leq \norm{u_{\vec{r}}.e_2} \leq CT$ for some $C>1$ dependent only on $\ep_0$ and $\Gamma$, we conclude
\[
   \int \varphi(u_{\vec{r} }.y) \, \diff\rmm_{z}(y) =    \int \varphi (y) \,\diff \rmm_z^U(y) 
      + O \big(  \Lip(\varphi)\norm{u_{\vec{r}}.e_2}^{-\delta}
      \big)
\]  
with $\delta:= 0.1\delta_1$.
The proof of Lemma \ref{lemma_focusing_1} is now complete.

\subsubsection{KHU decomposition}

Recall $K_0 = \SO_2(\R)^{ l_1} \times \SU_2(\R)^{ l_2}$ is a maximal compact subgroup of $G$.
To deduce Proposition \ref{proposition_equidistribution_focusing_affine_quadric} from Lemma \ref{lemma_focusing_1}, we need the following group theoretical result.
\begin{lem}\label{lemma_KHU_decomposition}
The maps $K_0 \times H \times U \to G$ and $K_0 \times H \times U^- \to G$  are both surjective.
\end{lem}

\begin{proof}
For simplicity we only prove the first one. So take $g=(g_1,...,g_{l_1+l_2})\in G$. 
For each $i=1,...,l_1+l_2$, there exist  $\lambda_i >0$  and $b_i \in \SO_2(\R)$ or $\SU_2(\R)$ depending on $i\leq l_1$ or $i >l_1$ such that 
\[
    b_i g_i .e_1 = \lambda_i e_1.
\]
Let $\lambda := \prod_{i=1}^{l_1} \lambda_i \cdot \prod_{i=1}^{l_2} \lambda_{l_1+i}^2$ and $\vec{s}:=(s_1,...,s_{l_1+l_2})$ with $s_i := \lambda^{\frac{1}{l}}/\lambda_i$ for $i=1,...,l_1+l_2$.
Then $h_{\vec{s}} \in H^{(1)}$ and 
\[
   h_{\lambda^{-1/l}} h_{s_i} b_i g_i .e_1 = e_1 ,\quad \forall\, i=1,...,l_1+l_2.
\]
Let $b := (b_1,...,b_{l_1+l_2}) \in K_0$.
Since the stabilizer of $e_1$ in $G$ is exactly $U$, we have $g = b^{-1} h_{\vec{s}}^{-1} h^{\Delta}_{\lambda^{1/l}}  u$ for some $u\in U$.
\end{proof}

\subsubsection{Proof of Proposition \ref{proposition_equidistribution_focusing_affine_quadric}  }
\label{subsubsection_proof_proposition_equidistribution_focusing_affine_quadric}
Before the proof, let us fix some $\kappa_1 >1$ such that  for all $s>0$
\[
   \Lip ( (h_s^{\Delta})^{-1}.\varphi) < \max\{s, s^{-1}\}^{\kappa_1 } \Lip(\varphi)
\]
where $ (h_s^{\Delta})^{-1}.\varphi(x):= \varphi( (h_s^{\Delta}).x)$.

For simplicity we only prove the part when $ x_1 > x_2^{\delta}$. Of course $\delta $ will only be determined later.
Fix $\varphi \in C_c^{\infty}(Y)$ supported on $\scrC$ and $(x_1,x_2,a)\in \calN_0^{\circ}$.
Find an $M=M_{\scrC}>0$ such that the support of $f_{\varphi}^z$ is contained in  $ [0,M]^2 \times \overline {A}$.
We further assume $x_1,x_2 \in (0,M]$.

Let $t\in \R^+$ be the unique number satisfying (see Eq.(\ref{equation_definition_x(t)}))
\[
    x_i^{-1} = \norm{ah_t^{\Delta}.e_i}_{k_{\infty}} , \;\text{ for } i=1,2,\;
    \text{ or equivalently, }   \;  t^l = \frac{1}{x_1 \norm{a.e_1}}_{k_{\infty}} = x_2 \norm{a.e_2}_{k_{\infty}}.
\]
So by definition (see Eq.(\ref{equation_f_quadratic}, \ref{equation_f_boundary_quadratic})),
\[
     f_{\varphi}(x_1,x_2,a) = \int \varphi(ah_t^{\Delta}.y)\, \diff \rmm_{\ep_0}(y),\quad
     f_{\varphi}(x_1,0,\infty) = \int \varphi(h^{\Delta}_{x_1^{1/l}} .y) \diff \rmm_{\ep_0}^U(y).
\]
Also, $  x_1 =  \norm{   b h_s^{\Delta} u_{\vec{r}} h^{(1)} .  e_1}_{k_{\infty}}^{-1} = s^{-l}
$.
On the other hand, by Lemma \ref{lemma_KHU_decomposition}, find $b \in K$, $h^{(1)} \in H^{(1)}$, $s\in \R^+$ and $\vec{r} \in k_{\infty}$ such that
\[
  a h_t^{\Delta} = b h_s^{\Delta} u_{\vec{r}} h^{(1)}.
\]
Now we transport our conditions into these new notations.

Firstly, as $\rmm_{\ep_0}$ is $H^{(1)}$-invariant and $\rmm_{\ep_0}^U$ is $UH^{(1)}$-invariant,
\begin{equation*}
\begin{aligned}
          f_{\varphi}(x_1,x_2,a) = \int \varphi(  h_s^{\Delta} u_{\vec{r}} . y)\, \diff \rmm_{\ep_0}(y),
          \quad
         f_{\varphi}(x_1,0,\infty) = \int \varphi(h^{\Delta}_{s} .y) \diff \rmm_{\ep_0}^U(y).
\end{aligned}
\end{equation*}
Secondly,
\begin{equation*}
\begin{aligned}
    &s^{-l} = x_1 \leq M \implies s^{-1} \leq M^{\frac{1}{l}},
    \\
    &
   x_1 > x_2^{\delta} \iff s^l  < \norm{ h_s^{\Delta} u_{\vec{r}}.e_2 }_{k_{\infty}}^{\delta}
   \leq \big(\max\{s^{2l}, s^{-2l} \}  \norm{  u_{\vec{r}}.e_2 }_{k_{\infty}} \big)^{\delta}
   \leq \big( s^{2l}M^2  \norm{  u_{\vec{r}}.e_2 }_{k_{\infty}} \big)^{\delta}
   . 
\end{aligned}
\end{equation*}
The latter one further implies that 
\begin{equation*}
\begin{aligned}
        &
        s ^{l(-1+\delta^{-1})} \leq M^2  \norm{  u_{\vec{r}}.e_2 }_{k_{\infty}} 
        \\
        \implies &
        s \leq \big(
        M^2  \norm{  u_{\vec{r}}.e_2 }_{k_{\infty}} 
        \big) ^{\frac{\delta}{l(1-\delta)}} \leq 
        \big(
        M^2  \norm{  u_{\vec{r}}.e_2 }_{k_{\infty}} 
        \big) ^{\frac{2\delta}{l}  }. 
\end{aligned}
\end{equation*}
Write the $\delta$ from Lemma \ref{lemma_focusing_1} as $\delta_1$ and let $\delta:= \frac{l}{4\kappa_1}\delta_1$. 
We conclude that 
\begin{equation*}
\begin{aligned}
  \max\{s,s^{-1}\}^{\kappa_1} 
  \leq & \max
   \Big\{
  M^{\frac{4
  \delta \kappa_1}{l}}  \norm{u_{\vec{r}} .e_2 }_{k_{\infty}}^{ \frac{2\delta\kappa_1 } {  l}  } ,
  M^{\frac{\kappa_1}{l}}
  \Big\}
  \\
  \leq &
  M^{\frac{\kappa_1}{l}} \norm{u_{\vec{r}}.e_2 }_{k_{\infty}}^{ \frac{2\delta\kappa_1  }{l} } 
  = M^{\frac{\kappa_1}{l}}
   \norm{u_{\vec{r}}.e_2 }_{k_{\infty}}^{ \frac{\delta_1 }{2l} } .
\end{aligned}
\end{equation*}
Then, applying Lemma \ref{lemma_focusing_1},  for every $z\in [H^{(1)}]$,
\begin{equation}
\begin{aligned}
    \normm{
     \int \varphi( h_s^{\Delta} u_{\vec{r}} y ) \, \diff \rmm_z(y) -  \int \varphi( h_s^{\Delta} y ) \, \diff \rmm^U_z(y)
    }
     & \leq \max\{s,s^{-1}\}^{\kappa_1} \Lip (\varphi) \norm{u_{\vec{r}} .e_2}_{k_{\infty}}^{-\delta_1}
     \\
     & \leq
     M^{\frac{\kappa_1}{l}}  \norm{u_{\vec{r}}.e_2 }_{k_{\infty}}^{- (1- \frac{1}{2l}  )} 
     \Lip(\varphi)
     \\
     & \leq M^{\frac{\kappa_1}{l}} \norm{u_{\vec{r}} .e_2 }_{k_{\infty}}^{ - 0.5 \delta_1} \Lip(\varphi).
\end{aligned}
\end{equation}
The proof is complete by integrating over $z\in [H^{(1)}]$ and taking $C:= M^{\frac{\kappa_1}{l}}$, $\delta':=   0.5 \delta_1$.

\subsection{Effective equidistribution: the generic case}\label{2.5}

The purpose of this subsection is to prove the following effective equidistribution statement:
\begin{prop}\label{proposition_effective_equidistribution_generic_quadric}
    For every $\delta  \in (0, 0.5)$, there exists $\delta'\in (0,0.5)$ such that the following holds. For every compact subset $\scrC\subset Y$, there exists
     $C>1$ such that for 
     every $(x_1,x_2,a)\in \calN_0^{\circ}$ and
    every smooth function $\varphi$ supported on $\scrC$ that is $K_0$-invariant, one has
   \[
    \normm{
        f_{\varphi}(x_1,x_2, a ) - f_{\varphi}(0,0,\infty)
    } \leq C \calS(\varphi) \max\{x_1,x_2\}^{\delta' }
    \]
    whenever $0<x_1<x_2^{\delta}$ and $0<x_2<x_1^{\delta}$.
\end{prop}

\subsubsection{$KAU$ decomposition}

\begin{lem}\label{lemma_KAU_decomposition_1}
    For every $g\in \SL_2(\R)$ satisfying $\norm{g.e_1}\geq 1$, there exist $k\in \SO_2(\R)$, $\lambda \in \R^+$ and $r\in \R$ such that $g= ka_{\lambda}u_r$. Similarly, $g$ can be written as $ ka_{\lambda}u_r^-$ if $\norm{g.e_2}\geq 1$ instead.
\end{lem}

\begin{proof}
    For simplicity we only prove the first part. It suffices to show that $g.e_1 = k a_{\lambda}.e_1$ for some $k\in \SO_2(\R)$ and $\lambda\in \R^+$.
    Since the orbit
    \[
    \{a_{\lambda}.e_1 ,\; \lambda\in \R^+ \}
    = \{ xe_1 +ye_2 ,\;
    x^2-y^2 =1 ,\, x>0
    \}
    \]
    is one component of a hyperbola, it suffices to find $k$ such that $k^{-1}g.e_1$ intersects with this (component of) hyperbola, which does exist since $\norm{g.e_1}\geq 1$.
\end{proof}

This lemma will be applied to $g=a_s h_t$ for certain $s,t\in \R^+$, after which we want some lower bound of the size of $r$.

\begin{lem}\label{lemma_KAU_decomposition_2}
    Given $s,t,\lambda\in \R^+$, $r\in \R$ and $k\in \SO_2(\R)$ such that $a_s h_t = k a_{\lambda} u_r$, we have  $\normm{r}+1 \geq t^{-2} $ and in particular $t \leq 11^{-1} \implies \normm{r} \geq 100$.
    Similarly, if $a_s h_t = k a_{\lambda} u^-_r$, we have  $\normm{ r } +1 \geq t^{2} $ and  $t \geq 11 \implies \normm{r} \geq 100$.
\end{lem}

\begin{proof}
    Again for simplicity we only present the proof of the first part.
    Let $a_s h_t = k a_{\lambda} u_r$ act on $e_1$, $e_2$, we get
    \begin{equation*}
        \begin{aligned}
            &
            t \norm{a_s .e_1} = \norm{a_{\lambda}.e_1 },
            \\
            &
            t^{-1}\norm{a_s .e_2} = \norm{ka_{\lambda}u_r.e_2}
            = \norm{ra_{\lambda}.e_1 + a_{\lambda}.e_2}
            \leq (\normm{r}+1)\norm{a_{\lambda}.e_1}
            \\
            &\implies
            t^{-2} \norm{a_s .e_2} \leq (\normm{r}+1) \norm{a_s.e_1} 
            \\
            &\implies
            t^{-2} \leq \normm{r} +1
        \end{aligned}
    \end{equation*}
    where we have used $ \norm{a_s.e_1}= \norm{a_s.e_2}$.
\end{proof}

\subsubsection{Thickening}
The purpose of this subsection is to make preparation for the thickening argument.
Here is the rough idea when $l_1+l_2=1$. Otherwise one could think of the $s,t ,\lambda, r$ appearing below are $s_i,t_i,\lambda_i,r_i$ for some $i\in \{1,...,l_1+l_2\}$.

By writing $a_s h_t= k a_{\lambda} u_r$, we will break its action on $ \rmm_z$ or $\rmm_{\ep_0}$ into two steps. First one translates by $u_r$, after which (if $r$ is large) one gets a rather long embedded manifold in the direction $u_rh_{\R^+}u_{r}^{-1}$  that stays inside a bounded region of  $Y$.  To apply mixing, we thicken the manifold  by the central and stable direction with respect to $a_{\lambda}$-action. To make this work, we note that the central-stable direction stays away from the direction of the submanifold. 

More notations have to be introduced. Let $k_{45^{\circ}} : = 
\frac{1}{\sqrt{2}}
\begin{bmatrix}
    1 & -1\\
    1 & 1
\end{bmatrix}
$.  For $\star, \varstar \in \{+, -\}$,
\begin{itemize}
    \item $\displaystyle 
    a_t = k_{45^{\circ}} h_t  k_{45^{\circ}} ^{-1}
    =\frac{1}{2}
    \begin{bmatrix}
    t+t^{-1} & t-t^{-1}  \\
    t-t^{-1} & t+t^{-1}
\end{bmatrix}
    $ and define $v_r^{\star}:=k_{45^{\circ}} u^{\star}_r  k_{45^{\circ}} ^{-1} $.
    \item 
    $\mathfrak{U}^{0+}:=
    \left\{
      \begin{bmatrix}
          x & y \\
          0 & -x
      \end{bmatrix} \midd x,y\in \R
    \right\}
    $, $\mathfrak{U}^{0-}:=
    \left\{
      \begin{bmatrix}
          x & 0 \\
          y & -x
      \end{bmatrix} \midd x,y\in \R
    \right\}
    $
    and $\mathfrak{V}^{0\star}:= k_{45^{\circ}} \mathfrak{U}^{0\star}  k_{45^{\circ}} ^{-1}$.
\end{itemize}
Thus $\mathfrak{V}^{0\star}$ is the stable/unstable direction of the $a_{\lambda}$-action.
Furthermore, for $r\in \R$ satisfying $\normm{r}\geq 100$, we define  
\begin{itemize}
    \item $ \displaystyle
    w^{\star}_r:= \frac{1}{ \lfloor |r| \rfloor  } u^{\star}_r 
       \begin{bmatrix}
              1  & 0 \\
               0 &  -1
        \end{bmatrix}
      (u^{\star}_r)^{-1} $. Concretely, 
      $
      w^+_r=  \begin{bmatrix}
         \frac{1}{  \lfloor |r| \rfloor}  &  -\frac{2r}{  \lfloor |r| \rfloor}  \\
        0 &  -\frac{1}{  \lfloor |r| \rfloor}
     \end{bmatrix}
      $ and $
      w^-_r=  \begin{bmatrix}
         \frac{1}{  \lfloor |r| \rfloor}  & 0  \\
         -\frac{2r}{  \lfloor |r| \rfloor} &  -\frac{1}{  \lfloor |r| \rfloor}
     \end{bmatrix}
      $.
\end{itemize}

Let $d(-,-)$ be the distance function on $\mathfrak{sl}_2(\R)$ induced from the Euclidean metric: $(A,B)\mapsto \tr(AB^{\tr})$.
An explicit calculation yields the following:
\begin{lem}\label{lemma_wave_front}
    For every real number $r$ satisfying $\normm{r}\geq 100$ and every $\star, \varstar \in \{+,-\}$, we have 
    $d(w_r^{\star}, \mathfrak{V}^{0\varstar}) \geq 0.1$.
\end{lem}

\begin{proof}
For simplicity we only prove the case when $\star= \varstar=+$ and $r>0$. The other cases can be handled in a similar way.

First we note that 
\begin{equation*}
\begin{aligned}
     d \Big(w_r^+ \,,\,
    \begin{bmatrix}
    0 & -2 \\
    0 & 0
    \end{bmatrix}
     \Big) = \sqrt{
     2 ( \lfloor \normm{r} \rfloor )^{-2}
     + 2^2 (  \frac{r}{\lfloor \normm{r} \rfloor} -1
     )^2
     } \leq \sqrt{
     \frac{2}{99^2}  + \frac{4}{99^2}
     } =\sqrt{6}/99 < 0.1.
\end{aligned}
\end{equation*}
Every element in $\mathfrak{V}^{0+}$ takes the form
\begin{equation*}
\begin{aligned}
    \begin{bmatrix}
    1 & -1 \\
    1 & 1
    \end{bmatrix}
    \begin{bmatrix}
    x & y \\
    0 & -x
    \end{bmatrix}
    \begin{bmatrix}
    1 & 1 \\
    -1 & 1
    \end{bmatrix}
    =
     \begin{bmatrix}
    -y & 2x+y \\
    2x-y &y
    \end{bmatrix}
\end{aligned}
\end{equation*}
for some $x,y\in\R$.
Note that 
\[
   -(2x-y) - 2 y + (2x+y+1) =1 \implies \min \{ 
   |2x-y|, |y|, |2x+y+1|
   \} \geq \frac{1}{6}.
\]
Therefore, for every $x,y \in \R$,
\begin{equation*}
\begin{aligned}
      d \Big(
    \begin{bmatrix}
    0 & -1 \\
    0 & 0
    \end{bmatrix}\,,\,
    \begin{bmatrix}
    -y & 2x+y \\
    2x-y &y
    \end{bmatrix}
     \Big) =\sqrt{
     (2x-y)^2 +2y^2 +(2x+y+1)^2
     } \geq 1/6,
\end{aligned}
\end{equation*}
and consequently,
\begin{equation*}
\begin{aligned}
       d \Big(
    \begin{bmatrix}
    0 & -2 \\
    0 & 0
    \end{bmatrix}\,,\,
    \mathfrak{V}^{0+}
     \Big) \geq 1/3 
     \implies 
      d \big(
    w_r^+\,,\,
    \mathfrak{V}^{0+}
     \big)  \geq 1/3-0.1 \geq 0.1.
\end{aligned}
\end{equation*}
So we are done.
\end{proof}

\subsubsection{Local coordinates}\label{subsubsection_local_coordinates}

For $r\in \R$, $\star, \varstar \in \{+ ,-\}$, $\eta\in \C$ and $\ep_0>0$ sufficiently small, define
\begin{itemize}
    \item $
     \displaystyle
    h^{r,\varstar}_{e^{\eta}} := 
     \begin{cases}
          \exp( \eta \cdot w_r^{\varstar})  = u_{r}^{\varstar } h_{ \eta/ { \left\lfloor  \normm{r} \right\rfloor
          }}    ( u_{r}^{\varstar })^{-1}
          & \text{ if }      \normm{r}\geq 100
        \\
         h_{e^{\eta}} & \text{ if } \normm{r} <100
    \end{cases}$;
    \item local charts about the identity
    \begin{equation*}
        \begin{aligned}
             \Phi_{r,\R}^{\star,\varstar} : (-\ep_{0},\ep_0)^3 &\to \SL_2(\R) \\
             (\theta, \gamma,\eta) &
             \mapsto a_{e^{\theta}} \cdot v_{\gamma}^{\star} \cdot h^{r,\varstar}_{e^{\eta}}   ,
             \\
             \Phi_{r,\C}^{\star,\varstar} : ( (-\ep_{0},\ep_0)+i (-\ep_{0},\ep_0))^3 &\to \SL_2(\C) \\
             (\theta, \gamma,\eta) &
             \mapsto a_{e^{\theta}} \cdot v_{\gamma}^{\star} \cdot h^{r,\varstar}_{e^{\eta}}   ;
        \end{aligned}
    \end{equation*}
    \item For $j\in \{1,...,l_1+l_2\}$, we let $\Phi_{r,j}^{\star,\varstar}$ be either $\Phi_{r,\R}^{\star,\varstar}$ or $\Phi_{r,\C}^{\star,\varstar}$ depending on whether the corresponding  embedding is real or complex.
    Putting them together, define for $\vec{r}= (r_1,...,r_{l_1+l_1}) \in \R^{l_1+l_2}$,
    \[
    \Phi^{\vec{\star},\vec{\varstar}}_{\vec{r}}:=
    \prod_{j=1}^{l_1+l_2} \Phi_{r,j}^{\star,\varstar}:
    (-\ep_{0},\ep_0)^l \times (-\ep_{0},\ep_0)^l \times (-\ep_{0},\ep_0)^l \to G.
    \]
    For $0<\ep\leq \ep_0$, let $ \calN^{\vec{\star},\vec{\varstar}}_{\vec{r}}(\ep)$ be the image of  $(-\ep,\ep)^l \times (-\ep,\ep)^l \times (-\ep_{0},\ep_0)^l $
    where the first two coordinates put restrictions on the size of $\vec{\theta}$ and $\vec{\gamma}$.
\end{itemize}

Let $\scrC_0$ be the bounded subset of $Y$ defined by
\[
   \scrC_0 := (U \cup U^- ) \cdot \{ h_{\vec{t }} \,  ,\; 11^{-1} \leq |t_j| \leq 11 ,\; \forall\, j  \} \cdot H^{(1)} \Gamma/\Gamma .
\]
The following is deduced from Lemma \ref{lemma_wave_front} and the wavefront property. The proof is omitted.
\begin{lem}
    For  $\ep_0, \eta_0 \in (0,0.5)$ sufficiently small, the following holds.
    For every $\vec{r} \in \R^{l_1+l_2}$ , $\vec{\star},\vec{\varstar}\in \{+,-\}^{l_1+l_2}$ and $z\in \scrC_0$, the map
    \begin{equation*}
        \begin{aligned}
            (-\ep_{0},\ep_0)^l \times (-\ep_{0},\ep_0)^l \times (-\ep_{0},\ep_0)^l &\to Y=G/\Gamma
            \\
            (\vec{\theta}, \vec{\gamma}, \vec{\eta})
            & \mapsto
            \Phi^{\vec{\star},\vec{\varstar}}_{\vec{r}}(\vec{\theta}, \vec{\gamma}, \vec{\eta}).z
        \end{aligned}
    \end{equation*}
    is a diffeomorphism onto some open neighborhood $ \calN^{\vec{\star},\vec{\varstar}}_{\vec{r}} (z) $ of $z$. 
\end{lem}

We fix, till the end of the proof of Proposition \ref{proposition_effective_equidistribution_generic_quadric},  $C_1>1$ such that for every $\vec{r} = (r_1,...,r_{l_1+l_2}) \in \R^{l_1+l_2}$, $\star, \varstar \in \{+, - \}^{l_1+l_2}$ and $\vec{\theta}, \vec{\gamma},\vec{\eta} \in (-\ep_0,\ep_0)^l$, one has
 \begin{equation}\label{equation_bound_distance}
\begin{aligned}
      d_{G}\left(   \id \,,  \Phi^{\vec{\star},\vec{\varstar}}_{\vec{r}}(\vec{\theta}, \vec{\gamma}, \vec{\eta})
      \right) \leq C_1 \max \left\{ \vert{\theta_j}\vert, \,\vert{\gamma_j} \vert ,\, \vert{\eta_j} \vert,\; j=1,...,l_1+l_2 
      \right\}.
\end{aligned}
\end{equation}

Since $\rmm_G$ is right $G$-invariant, we find $\rho^{\vec{\star},\vec{\varstar}}_{\vec{r}}: (-\ep_{0},\ep_0)^l \times (-\ep_{0},\ep_0)^l \to [0,+\infty)$ such that 
\begin{equation}\label{equation_density_function_volume}
\left( \Phi^{\vec{\star},\vec{\varstar}}_{\vec{r}}\right)_*
\left(
\rho^{\vec{\star},\vec{\varstar}}_{\vec{r}}(\vec{\theta},\vec{\gamma}) \diff \vec{\theta} \diff \vec{\gamma} \diff \vec{\eta}  
\right) = \rmm_G \big\vert_{\calN^{\vec{\star},\vec{\varstar}}_{\vec{r}}}.
\end{equation}
We take $C_2>1$ such that $C_2^{-1}\leq \normm{\rho^{\vec{\star},\vec{\varstar}}_{\vec{r}}} \leq C_2$ on $(-\ep_0,\ep_0)^{2l}$.

\subsubsection{Approximation by smooth functions}\label{subsubsection_approximate_by_smooth_functions}

For $0<\ep<\ep_0$, one can find  smooth functions $ \beta_{\ep}, \, \beta_{\ep}^{\ep_0}: \R\to [0,1]$ enjoying the following properties:

First, $\supp(\beta_{\ep}) \subset (-\ep,\ep)$,  $\supp(\beta^{\ep_0}_{\ep}) \subset (-\ep_{0},\ep_0)$.

Next, by abuse of notation, define for $\vec{s}\in \R^l$, 
\[
   \beta_{{\ep}}(\vec{s}) := \prod_{i=1}^l
   \beta_{\ep}(s_i),\quad
   \beta^{\ep_0}_{{\ep}}(\vec{s}) := \prod_{i=1}^l
   \beta^{\ep_0}_{\ep}(s_i),
\]
then
\[
    \normm{
    \frac{1}{ (2\ep_0)^l(2 \ep)^{2l}  }
    \int_{(-\ep_0,\ep_0)^l} \int_{(-\ep,\ep)^{2l}} \beta_{\vec{\ep}}(\vec\theta) \beta_{\vec{\ep}}(\vec\gamma) \beta_{\vec{\ep}}^{\ep_0}(\vec{\eta})
    \,
    \diff \vec{\theta} \diff \vec{\gamma} \diff \vec{\eta}
    -  1
    } \leq \ep.
\]
Moreover,
let  $\beta^{\vec{\star},\vec{\varstar}}_{\vec{r},z,\ep}:Y\to \R$ be the function such that
$\supp(\beta^{\vec{\star},\vec{\varstar}}_{\vec{r},z,\ep}) \subset \calN^{\vec{\star},\vec{\varstar}}_{\vec{r}}(z)$ defined by
\begin{equation*}
    \beta^{\vec{\star},\vec{\varstar}}_{\vec{r},z,\vec{\ep}} \left(
    \Phi^{\vec{\star},\vec{\varstar}}_{\vec{r}}
    (\vec\theta,\vec\gamma,\vec\eta).z
    \right)
   := \beta_{\vec{\ep}}(\vec\theta) \beta_{\vec{\ep}}(\vec\gamma) \beta_{\vec{\ep}}^{\ep_0}(\vec\eta).
\end{equation*}
Finally we fix, until the end of the proof of Proposition \ref{proposition_effective_equidistribution_generic_quadric},
$\kappa_1>1$ such that for all  $z\in \scrC_0$,  $0<\ep<\ep_0$ and $\bmr$,
   \[
      \calS\left(\beta^{\vec{\star},\vec{\varstar}}_{\vec{r},z,\ep}  \right ) \leq \ep^{-\kappa_1}.
   \]

\subsubsection{Proof of Proposition \ref{proposition_effective_equidistribution_generic_quadric}}
Let $\delta \in (0,0.5)$, $\scrC\subset Y$, $(x_1,x_2,a)\in \calN_0^{\circ}$ and $\varphi$ be given.
By assumption $x_{2}^{1/\delta} \leq x_1 \leq x_2^{\delta}$. 
Let 
\[
 t := \frac{1}{
 x_1^{ 1/l  }   \norm{a.e_1}_{k_{\infty}} ^{ 1/l }
 }=
 x_2^{ 1/l }   \norm{a.e_2}_{k_{\infty}} ^{ 1/ l  }.
\]
Then
\begin{equation*}
\begin{aligned}
      x_1 = t^{-l} \norm{a.e_1}_{k_{\infty}}^{-1}=  \norm{ah^{\Delta}_t.e_1}_{k_{\infty}}^{-1}
      , \quad
      x_2 = t^{l} \norm{a.e_2}_{k_{\infty}}^{-1}= \norm{ah^{\Delta}_t.e_2}_{k_{\infty}}^{-1}
\end{aligned}
\end{equation*}
and
\begin{equation}\label{equation_f_varphi}
  f_{\varphi}(x_1,x_2,a ) = 
   \frac{1}{ (2\ep_0)^l } \int_{[H^{(1)}]} \int_{ (-\ep_{0},\ep_0)^l} \varphi
   \left( a_{\vec{s}} h_{\vec{t}} h_{e^{\vec{\xi}}}.z
   \right) \, \diff \vec{\xi} \diff \rmm_{[H^{(1)}]}(z) .
\end{equation}
Assume without loss of generality that $x_1\leq x_2$ or equivalently $t\geq 1$.

We want to make the situation more ``balanced''.
Write $a= a_{\vec{s}}$ with $\vec{s} = (s_1,...,s_{l_1+l_2}) \in (\R^+)^{l_1+l_2} \subset k_{\infty}$
and define $\vec{t}:=(t_1,...,t_{l_1+l_2})
 \in (\R^+)^{l_1+l_2} \subset k_{\infty} $ by
\[
   \frac{\log t_j }{\log \norm{a_{s_j }.e_1}}
   = \frac{\log (t^l)}{
   \log \norm{a_{\vec{s}}. e_1}_{k_{\infty}}
   }, \quad  j =1,...,l_1+ l_2.
\]
If $\delta_1 > 0$ is defined by $t^l = \norm{a_ {\vec{s}} . e_1}^{\delta_1}_{k_{\infty} }$, then $t_j = \norm{a_{s_j }.e_1}^{\delta_1}\geq 1$ for every $j =1,...,l_1+l_2$.
Thus, $\norm{\vec{t}} _{k_{\infty}} = t^l$ and $h^{\Delta}_t = h_{\vec{t}} \cdot h^{(1)}$ for some $h^{(1)} \in H^{(1)}$. 
The condition $x_{2}^{1/\delta} \leq x_1 \leq x_2^{\delta}$ is equivalent to 
\begin{equation*}
\begin{aligned}
    \norm{a.e_2}_{k_{\infty}}^{\delta(1-\delta_1)} \leq \norm{a.e_1}_{k_{\infty}}^{1+\delta_1} \leq \norm{ a.e_2 }_{k_{\infty}}^{\delta^{-1}(1-\delta_1)  }.
\end{aligned}
\end{equation*}
We may assume $a \neq \id$, in which case $  \norm{a.e_1}_{k_{\infty}} =   \norm{a.e_2}_{k_{\infty}} >1$. So the above is equivalent to
\begin{equation*}
\begin{aligned}
    \delta_1 \leq \frac{1-\delta}{1+\delta}.
\end{aligned}
\end{equation*}

Take $j\in \{1,..., l_1+l_2\}$. 
The ways of thickening are different depending on $\normm{ t_j } >11 $ or not.

\textit{Case 1, $\normm{t_j} \leq 11 $.}
We formally set $r_j:=1$, $\lambda_j:= s_j$ and for $\tau_j \in (-\ep_0,\ep_0)$ (if the place is real),
\begin{equation*}
     g_j:= a_{\lambda_j},\quad
     M_j(\tau_j ):=h_{e^{\tau_j}}= h^{r_j-}_{e^{\tau_j}}, \quad
     B_{j,0}   :=  h_{t_j}.
\end{equation*}
So
\[
     \left\{ a_{s_j} h_{t_j} h_{e^{\xi_j}}  ,\; \xi_j \in (-\ep_0,\ep_0) \right\}
    =
    \left\{ g_j M_j(\tau_j) B_{j,0}  ,\; \tau_j \in (-\ep_0,\ep_0) \right\}.
\]
If the place is complex, one replaces $(-\ep_0,\ep_0)$ by $(-\ep_0,\ep_0)+ i(-\ep_0,\ep_0)$ instead.

\textit{Case 2, $\normm{t_j} >11 $.}
Note that 
\[
 \norm{a_{s_j} h_{t_j}.e_2 
 } = t_j^{-1}\norm{a_{s_j}.e_2 } = \norm{a_{s_j}.e_2}^{1-\delta_1} \geq 1.
\]
Apply Lemma \ref{lemma_KAU_decomposition_1} and \ref{lemma_KAU_decomposition_2} to write
\[
   a_{s_j} h_{t_j} = b_j a_{\lambda_j} u^-_{r_j},\quad \text{for some } b_j \in \SO_2(\R),\;  \lambda_j\in \R^+,\; r_j\in \R,\;\normm{r_j} \geq 100.
\]
For $n_j \in \{0,1,...,\lfloor |r_j| \rfloor\}$ and $\tau_j \in (-\ep_{0},\ep_0)$ (if the place is real),
set 
\begin{equation*}
    \begin{aligned}
            g_j:= a_{\lambda_j},
           \quad
        M_j(\tau_j):= h^{r_j-}_{e^{\tau_j}},
        \quad
            B_{j,n_j }:= u^-_{r_j} \cdot h_{e^
        { (2 \ep_0  n_j + \ep_0) ({\lfloor |r_j| \rfloor })^{-1} -  \ep_0}
        }.
    \end{aligned}
\end{equation*}
Then 
\begin{equation*}\label{equation_decompose_small_pieces}
    \left\{ a_{\lambda_j} u^{-}_{r_j} h_{e^{\xi_j } } \midd  \xi_j \in (-\ep_{0},\ep_0]
    \right\}
    =
    \bigsqcup_{n_j = 0 }^{\lfloor |r_j| \rfloor -1}
    \left\{ 
       g_j M_j(\tau_j)  B_{j,n_j} \midd \tau_j \in (-\ep_{0},\ep_0]
    \right\}
\end{equation*}
and the Lebesgue measure $\diff \xi_j$ becomes $\frac{1}{\lfloor |r_j| \rfloor} \sum_{n_j=0}^{\lfloor |r_j| \rfloor -1} \diff \tau_j$.
When the place is complex, $(-\ep_0,\ep_0)$ should be replaced by $(-\ep_0,\ep_0)+i(-\ep_0,\ep_0)$ and the range of $n_j$ should be $\{0,..., \lfloor |r_j| \rfloor -1\}^2$. To lighten the notations, which are already quite complicated, we spell the notations only for the real place in the discussion below with the understanding that modifications need to be made when the place is complex.

As usual, for $\vec{\tau}:= (\tau_1,...,\tau_{l_1+l_2})$ and $\vec{n}=(n_1,...,n_{l_1+l_2})$ with $0\leq n_j \leq \lfloor |r_j| \rfloor-1 $, let 
\[
   \vec{g}:= (g_1,...,g_{l_1+l_2}), \quad
   \vec{M}(\vec{\tau}):=  (M_1 (\tau_1),...,M_{l_1+l_2}(\tau_{l_1+l_2 })  ),\quad
   \vec{B}_{\vec{n}} :=  (B_{1,n_1},...,B_{l_1+l_2, n_{l_1+l_2}}).
\]
With these notations, update Eq.(\ref{equation_f_varphi})  as
\begin{equation}\label{equation_f_varphi_1}
\begin{aligned}
    &
      f_{\varphi}(x_1,x_2,a ) 
      \\
   = \,&
   \frac{1}{\prod_{j=1}^{l_1+l_2} \lfloor |r_j| \rfloor}
    \sum_{\vec{n},\; n_j =0,... \lfloor |r_j| \rfloor-1 
   }
   \int_{  [H^{(1)}]   }
   \int_{\vec{\tau}\in (-\ep_{0},\ep_0)^l}
   \varphi(  \vec{g} \vec{M}(\vec{\tau}  )\vec{B}_{\vec{n}} .z ) \, \diff \vec\tau \, \diff \rmm_{[H^{(1)}]}(z)
   \\
   =:  \,&
   \frac{1}{\prod_{j=1}^{l_1+l_2} \lfloor |r_j| \rfloor}
   \sum_{\vec{n},\; n_j =0,... \lfloor |r_j| \rfloor-1 
   }
   \int_{  [H^{(1)}]   }  \scrI(x_1,x_2,a, \vec{n} ,z ) \, \diff \rmm_{[H^{(1)}]}(z).
\end{aligned}
\end{equation}
Now we perform thickening trick by inserting something in between $g$ and $\vec{M}(\vec{\tau})$.
Several notations here are defined in section \ref{subsubsection_local_coordinates} and \ref{subsubsection_approximate_by_smooth_functions}.

Define $\vec{\star}=(\star_j)$ by $\star_j := + $ iff $\lambda_j<1$ and is $-$ otherwise.
Therefore, given $\theta_j, \gamma_j \in (-\ep,\ep)$ for some $0<\ep<\ep_0$, there exists $ \vert \gamma'_j \vert \leq \vert \gamma_j \vert$ such that
\begin{equation*}
\begin{aligned}
     g_j a_{e^{\theta_j}} v^{\star_j}_{\gamma_j}     
      = a_{e^{\theta_j}} v^{\star_j}_{\gamma'_j} g_j
\end{aligned}
\end{equation*}
where one is reminded that $g_j= a_{\lambda_j}$.
Also recall that $a_{e^{\theta_j}} v^{\star_j}_{\gamma'_j}= \Phi_{r_j,j}^{\star_j, -}(\theta_j,\gamma'_j, 0)$ (actually $r_j$ and $-$ play no role here).
As $d_G(-,-)$ is right invariant and by Eq.(\ref{equation_bound_distance}), we have
\begin{equation*}
\begin{aligned}
     &
    d_G  \big(  \vec{g}, \vec{g} a_{e^{\vec{\theta}}} v_{\vec{\gamma}}^{\vec{\star}}    \big)
    = d_G   \big( \id ,   a_{e^{\vec{\theta}}} v^{\vec{\star}}_{\vec{\gamma}'} 
     \big) =
    d_G  \big(
    \id, \Phi_{\vec{r}}^{\vec{\star}, -}( \vec{\theta}, \vec{\gamma}', 0)
    \big) \leq C_1 \ep.
\end{aligned}
\end{equation*}
Writing $z_{\vec{n}} := \vec{B}_{\vec{n}}.z$, which lives inside $\scrC_0$, for symbols appearing in Eq.(\ref{equation_f_varphi_1}) and $\vec{\theta}, \vec{\gamma} \in (-\ep,\ep)^l$, one has
\begin{equation*}
\begin{aligned}
     d_Y \Big(
        \vec{g} \vec{M}(\vec{\tau}  )  .z_{\vec{n}} \,,\, 
         \vec{g}    \Phi_{\vec{r}}^{\vec{\star},-}(\vec{\theta}, \vec{\gamma}, \vec{\tau})   .z_{\vec{n}}
     \Big) 
     \leq d_G
     \Big(
        \vec{g} \vec{M}(\vec{\tau}  )  \,,\, 
         \vec{g}   a_{e^{\vec{\theta}}} v^{\star}_{\vec{\gamma}} \vec{M}(\vec{\tau} )
     \Big) 
     \leq C_1 \ep.
\end{aligned}
\end{equation*}
Inserting into Eq.(\ref{equation_f_varphi_1})  and integrating over $ \frac{ \rho_{\vec{r}}^{\vec{\star}, -} (\vec{\theta}, \vec{\gamma})}{  \rmm_G(\calN_{\vec{r}}^{\vec{\star}, -}(\ep)  )
}   \diff \vec{\theta} \diff \vec{\gamma} $, we get (recall that by Eq.(\ref{equation_density_function_volume}), $\rho_{\vec{r}}^{\vec{\star}, -} $ is the density function of the Haar measure in local coordinates)
\begin{equation*}
\begin{aligned}
       &\scrI (x_1,x_2,a, \vec{n},z ) 
        \\
   = \,&
   \int_{\vec{\tau}\in (-\ep_{0},\ep_0)^l} \int_{(-\ep,\ep)^{2l}}
    \varphi  \left(      \vec{g}    \Phi_{\vec{r}}^{\vec{\star},-}(\vec{\theta}, \vec{\gamma}, \vec{\tau})    .z_{\vec{n}}
    \right) \, \frac{  \rho_{\vec{r}}^{\vec{\star}, -} (\vec{\theta}, \vec{\gamma})}{  \rmm_G(\calN_{\vec{r}}^{\vec{\star}, -}(\ep)  )
}   \diff \vec{\theta} \diff \vec{\gamma}  \diff \vec\tau  + \bmO(C_1\ep \Lip(\varphi))
\\
  =\,&
   \int_{  \calN_{\vec{r}}^{\vec{\star}, -}(\ep)  .z_{\vec{n}}
   }
    \varphi(      \vec{g}  . y
    ) \, \frac{ 1  }{  \rmm_G(\calN_{\vec{r}}^{\vec{\star}, -}(\ep)  )
}  \diff \rmm_Y(y) + \bmO(C_1\ep \Lip(\varphi))
\\
=\,&
    \frac{ 1  }{  \rmm_G(\calN_{\vec{r}}^{\vec{\star}, -}(\ep)  ) }
     \int_{ Y
   }
    \varphi(      \vec{g}  . y
    )
    \beta_{\vec{r},z_{\vec{n}}, \ep  }^{\vec{\star}, -} (y)
     \,
   \diff \rmm_Y(y) + \bmO((C_1+C_2^2) \ep \Lip(\varphi)).
\end{aligned}
\end{equation*}
Let $\delta_2$ be the $\delta $ appearing in Theorem \ref{theorem_effective_mixing}.
We get 
\begin{equation}\label{equation_apply_mixing}
\begin{aligned}
       \scrI (x_1,x_2,a, \vec{n},z ) 
       = \int_{ Y }
      \varphi(  y ) \diff \rmm_Y(y)
      \frac{ \int_Y
        \beta_{\vec{r},z_{\vec{n}} , \ep }^{\vec{\star}, -} ( y )
      \, \diff \rmm_Y(y)
      }{     \rmm_G(\calN_{\vec{r}}^{\vec{\star}, -}(\ep)  )
      } + \frac{ \ep^{-\kappa_1} \calS(\varphi)  }{
       \rmm_G(\calN_{\vec{r}}^{\vec{\star}, -}(\ep)   )
      }
      (\norm{\vec{g}.e_1}_{k_{\infty}} \norm{\vec{g}.e_2}_{k_{\infty} }) ^{-\delta_2}.
\end{aligned}
\end{equation}
Using section \ref{subsubsection_approximate_by_smooth_functions}, one can check that 
\begin{equation*}
\begin{aligned}
       &
        \frac{1
      }{     \rmm_G(\calN_{\vec{r}}^{\vec{\star}, -}(\ep)  )
      }  \int_Y
        \beta_{\vec{r},z_{\vec{n}} , \ep }^{\vec{\star}, -} ( y )
      \, \diff \rmm_Y(y)
     \\
      = &\,
       \frac{1
      }{     \rmm_G(\calN_{\vec{r}}^{\vec{\star}, -}(\ep)  )
      }  \int_{(-\ep,\ep)^{2l}\times (-\ep_0,\ep_0)^l}
      \beta_{\ep}(\vec{\theta}) \beta_{\ep}(\vec{\gamma}) \beta_{\ep}^{\ep_0}(\vec{\eta})
      \rho_{\vec{r}}^{\vec{\star}, -} (\vec{\theta}, \vec{\gamma}) \, \diff \vec{\theta} \diff \vec{\gamma} \diff \vec{\tau}
      \\
      =&\,
       \frac{1
      }{     \rmm_G(\calN_{\vec{r}}^{\vec{\star}, -}(\ep)  )
      }  
       \int_{(-\ep,\ep)^{2l}\times (-\ep_0,\ep_0)^l}
       \big( 1 + \bmO((2\ep)^{2l}(2\ep_0)^l \ep) \big)  \rho_{\vec{r}}^{\vec{\star}, -} (\vec{\theta}, \vec{\gamma}) \, \diff \vec{\theta} \diff \vec{\gamma} \diff \vec{\tau}
       \\
       =&\,
      1+ \bmO(C_2^2 \ep).
\end{aligned}
\end{equation*}
Thus the first summand in Eq.(\ref{equation_apply_mixing}) is  equal to
\begin{equation*}
\begin{aligned}
      \int_{ Y }
      \varphi(  y ) \diff \rmm_Y(y) + \bmO(\Lip(\varphi)  C_2^2 \ep  ).
\end{aligned}
\end{equation*}

Next we turn to the second summand.
The definition of $\lambda_j$ was divided into two cases. In case $|t_j| \leq 11$, we have 
$    \norm{a_{\lambda_j} .e_{2}} \geq \frac{1}{11}   \norm{a_{s_j} h_{t_j}.e_{2}} $. In case $|t_j| >11$, we have 
$    \norm{a_{\lambda_j} .e_{2}} =  \norm{a_{s_j} h_{t_j}.e_{2}} $. Also note that $\norm{\vec{g}.e_1}_{k_{\infty}} = \norm{\vec{g}.e_2}_{k_{\infty}}$.
Therefore,
\begin{equation*}
\begin{aligned}
     \norm{\vec{g}.e_1}_{k_{\infty} }  \norm{\vec{g}.e_2}_{k_{\infty} }  & = \norm{\vec{g}.e_2}_{k_{\infty} } ^2 = \prod_{j=1}^{l_1+l_2} \norm{ a_{\lambda_j}.e_2}^2
     \\
     & \geq 
     \frac{1}{11^{2(l_1+l_2)}  } \prod_{j=1}^{l_1+l_2} \norm{ a_{s_j} h_{t_j}  .e_2}^2
     \\
     & =
      \frac{1}{11^{2(l_1+l_2)}  } \norm{ a h^{\Delta}_{t}  .e_2}_{k_{\infty} } ^2 = \frac{1}{11^{2(l_1+l_2)}  } x_2 ^{-2}.
\end{aligned}
\end{equation*}
So
\begin{equation*}
\begin{aligned}
     I(x_1,x_2, a, \vec{n}, z) = \int_Y \varphi(y) \,\diff\rmm_Y(y) + \bmO(\Lip(\varphi) C_2^2 \ep)
     +
     \frac{  C_2 11^{2\delta_2(l_1+l_2)  } }{(8\ep_0)^l}
     \bmO\Big(
     \frac{  \calS(\varphi) x_2^{2\delta_2}
     }{ \ep^{2l + \kappa_1}  }
     \Big).
\end{aligned}
\end{equation*}
Choose $\ep:= x_2^{\delta_2/ (2l+\kappa_1)}$, then
\begin{equation*}
\begin{aligned}
     I(x_1,x_2, a, \vec{n}, z) = \int_Y \varphi(y) \,\diff\rmm_Y(y) + C_2^2  \bmO \big(  \Lip(\varphi) x_2^{\delta_2/(2l+\kappa_1)}
     \big)
     +
     \frac{  C_2 11^{2\delta_2(l_1+l_2)  } }{(8\ep_0)^l}
     \bmO(
      \calS(\varphi) x_2^{\delta_2}
      ).
\end{aligned}
\end{equation*}
Inserting into Eq.(\ref{equation_f_varphi_1}), the proof is complete, at least when $x_2$ is sufficiently small so that $\ep$ just defined is smaller than $\ep_0$. Otherwise it is not hard to adjust the constant so that the conclusion also holds.

\appendix

\section{Volume asymptotics}\label{appendix_A}

We explain here the volume asymptotics that were needed for the proof of  Theorem \ref{theorem_smoothed_count_quadric} following \cite{Chamber-Loir_Tschinkel_2010_Igusa_integral}. Strictly speaking the case at hand is not covered by their results, but the method is similar.

Recall that $\bmX$ is defined by the affine quadric $q(x_1,x_2,x_3)=m$, which embeds into $\bmP^3$ via $(x_1,x_2,x_3) \to [1:x_1:x_2:x_3]$.
Let $\overline{\bmX}$ be its closure. Then both $\overline{\bmX}$ and the boundary $\bmD:= \overbmX \setminus \bmX$ are irreducible and smooth.
Let $\bmone_{\bmD}$ be the canonical section of the line bundle $\calL:= \calO_{\bmX}(\bmD)$. 
For each archimedean place $\nu$ of $k$, equip $\calL_{k_{\nu}}$  with a smooth metric $\norm{\cdot}_{\calL_{\nu}}$ such that 
\[
   \norm{(x_1,x_2,x_3)} = \norm{ \bmone_{\bmD} ([1:x_1:x_2:x_3])}_{\calL_{\nu}}^{-1},\quad \forall (x_1,x_2,x_3)\in k_{\nu}^3.
\]

Let $v \cdot w \in k_{\nu}^3$ be the symmetric product of two vectors in $k_{\nu}^2$.
So for some Euclidean metric on $\R^3$ or $\C^3$, $\norm{v\cdot w} = \norm{v}\norm{w}$. In particular, we can find possibly a different smooth metric  $\norm{\cdot}_{\calL'}$ on $\calL_{k_{\nu}}$ (for each archimedean place $\nu$ of $k$) such that 
\[
   T_g = \norm{g.e_1}_{k_{\infty}} \norm{g.e_2}_{k_{\infty}} =
    \prod_{i=1}^{l_1+l_2} \norm{ \bmone_{\bmD} ([1:g.x_0])}_{\calL_{\nu_i}'}^{-\epsilon_i}, \quad \forall \, g\in G.
\]
where $\epsilon_i =1$ if $i=1,...,l_1$ and $\epsilon_i=2$ if $i=l_1+1,...,l_1+l_2$.
Consider the height zeta functions
\begin{equation*}
\begin{aligned}
     Z_{\log}(\tau) &:= \int_{ G/H }  \log {T_g}  \norm{g.x_0}_{k_{\infty}}^{-\tau} \diff \rmm_{G/H}([g]),
     \\
    Z (\tau) &:= \int_{G/H } \norm{g.x_0}_{k_{\infty}}^{-\tau} \diff \rmm_{G/H}([g]).
 \end{aligned}
\end{equation*}
They converge when the real part of $\tau \in \C$ is sufficiently large and the asymptotic in Eq.(\ref{equation_smooth_count_quadratic}) can be read from their meromorphic continuations. For the sake of simplicity and also because it is not covered by  \cite{Chamber-Loir_Tschinkel_2010_Igusa_integral}, we only discuss the former one in details.

Let $\omega_X$ be the volume form inducing the Haar measure $\rmm_{G/H}$. Denote by $\normm{\omega }_{ \nu } $ the Haar measure it induces on $\bmX(k_{\nu})$. 
Recall that $\{\nu_1,...,\nu_{l_1+l_2} \}$ denotes the archimedean places of $k$ up to complex conjugation.
Thus 
\[
\rmm_{G/H} \cong \otimes_{i =1}^{l_1+l_2}  \normm{\omega }_{ \nu_i }.
\]
We want to remark here that when  $k_{\nu}$ is complex and $\omega_{\bmX}$ is locally given by $ \varphi(z_1,z_2)\diff z_1 \wedge \diff z_2$, the corresponding measure $\normm{\omega}_{\nu} $ is induced by integrating against $2 \normm{\varphi}^2 \diff x_1 \diff y_1 \diff x_2 \diff y_2 $ if $z_j = x_j + i y_j$ for $j =1,2$ (see \cite[2.1.1]{Weil_groups}).

Define for $i=1,...,l_1+l_2$, 
\begin{equation*}
\begin{aligned}
     Z_{\log,\nu_i}(\tau) &:= \int_{g.x_0 \in \bmX(k_{\nu_i})} 
     \log\left( \norm{ \bmone_{\bmD} ([1:g.x_0])}^{\epsilon_i}_{\calL_{\nu_i}'} \right)
      \norm{ \bmone_{\bmD} ([1:g.x_0])}_{\calL_{\nu_i}}^{-{\epsilon_i}\tau}\, \diff\! \normm{\omega }_{ \nu_i } (g.x_0),
     \\
    Z_{\nu_i}(\tau) &:= \int_{g.x_0 \in \bmX(k_{\nu_i})} 
      \norm{ \bmone_{\bmD} ([1:g.x_0])}_{\calL_{\nu_i}}^{- {\epsilon_i}\tau}\,\diff \!\normm{\omega}_{ \nu_i } (g.x_0) .
 \end{aligned}
\end{equation*}
We have
\begin{equation}\label{equation_height_zeta_product_places}
\begin{aligned}
       Z_{\log}(\tau) = \sum_{i=1}^{l_1+l_2}
       Z_{\log,\nu_i}(\tau) \cdot \prod_{j\neq i} Z_{\nu_j}(\tau) .
\end{aligned}
\end{equation}

\begin{lem}\label{lemma_height_zeta_meromorphic_continuation_quadratic}
Let notations be as above. $Z_{\log}(\tau)$ has the following property:
\begin{itemize}
     \item[(1)]  it has a meromorphic continuation to the half space $\mathrm{Re}(\tau)> 0.5$;
     \item[(2)]  it has  a pole of order $l_1+l_2+1$  at $\tau=1$ with positive residue and is analytic elsewhere;
     \item[(3)] its derivatives have moderate growth in (closed) vertical strips on the half space $\mathrm{Re}(\tau)> 0.5$.
\end{itemize}
The same thing is true for $Z(\tau)$ except that the order of pole is $l_1+l_2$.
\end{lem}

\begin{rmk}
It follows from Lemma \ref{lemma_height_zeta_meromorphic_continuation_quadratic} that for any $\delta>0$, the zeta function
\[
     Z_{\delta}(\tau):=   \int_{g.x_0 \in \bmX(k_{\infty})}  T_g^{-\delta} \norm{g.x_0}_{k_{\infty}}^{-\tau} \diff \rmm_{G/H}([g])
\]
is analytic on $\mathrm{Re}(\tau)>1-\delta'$ for some $\delta'>0$. Hence for some possibly different $\delta''>0$ and $C>1$
\[
\int_{g.x_0\in B_T} T_g^{-\delta}\diff \rmm_{G/H}(gH)   \leq C T^{1-\delta''}.
\]
\end{rmk}

\begin{proof}
In light of Eq.(\ref{equation_height_zeta_product_places}), it suffices to deal with $Z_{\log,\nu_i}(\tau)$ and $Z_{\nu_i}(\tau)$.
Let us first assume $i=1,...,l_1$, that is, $\nu_i$ is real. The complex case requires a slight modification, which will be explained later.

By partition of unity, there exists a finite open covering $( \calU_{\alpha}) $ of $\overbmX(\R)$ and $\phi_{\alpha} : \overbmX(\R) \to [0,1]$ supported in $\calU_{\alpha}$ with $\sum \phi_{\alpha} =1$.  Therefore,
\begin{equation*}
\begin{aligned}
     Z_{\log,\nu_i}(\tau) =&
      \sum_{\alpha} \int_{g.x_0 \in \bmX(k_{\nu_i})} \phi_{\alpha}(g.x_0)
     \log\left( \norm{ \bmone_{\bmD} ([1:g.x_0])}^{\epsilon_i}_{\calL_{\nu_i}'} \right)
      \norm{ \bmone_{\bmD} ([1:g.x_0])}_{\calL_{\nu_i}}^{-{\epsilon_i}\tau}\, \diff\! \normm{\omega }_{ \nu_i } (g.x_0)
      \\
      =:&
      \sum_{\alpha}  Z^{\alpha}_{\log,\nu_i}(\tau).
\end{aligned}
\end{equation*}
Furthermore, without loss of generality, let us assume that $\calU_{\alpha_0} \subset \bmX(\R)$ and  for ${\alpha}\neq \alpha_0$, $\calU_{\alpha}$ has a local chart
$\Psi_{\alpha}: (-1,1)^2 \cong \calU_{\alpha}$ such that $\Psi_{\alpha}^{-1}\left( \calU_{\alpha} \cap \bmD(\R) \right) = \{(x_1,x_2) ,\; x_2 =0 \}$.
Moreover, the local coordinate $x_2 \circ \Psi_{\alpha}^{-1}$ may be chosen to be a local generator of $\bmD$.
Finally, the pole of the volume form $\omega_X$ has order two along the boundary $\bmD$.
Therefore,
there exist smooth functions $u_{\alpha}, v_{\alpha}  :(-1,1)^2 \to [0,+\infty)$  such that
\begin{itemize}
     \item $v_{\alpha}$ has compact support and $v_{\alpha} (0,0)\neq 0$;
     \item  $u_{\alpha} (x_1,x_2)>0$ for all $(x_1,x_2)\in (-1,1)^2$;
     \item $\displaystyle
     Z^{\alpha}_{\log,\nu_i}(\tau) =  \int_{-1}^1 \int_{-1}^{1} \log \left( \normm{x_2}^{-1} u_{\alpha}(x_1,x_2) \right) \cdot \normm{x_2}^{\tau} \cdot v_{\alpha}(x_1,x_2) \cdot \frac{\diff x_1 \diff x_2}{\normm{x_2}^2};
     $
     \item $\displaystyle
     Z^{\alpha}_{\nu_i}(\tau) =  \int_{-1}^1 \int_{-1}^{1} \normm{x_2}^{\tau} \cdot v_{\alpha}(x_1,x_2) \cdot \frac{\diff x_1 \diff x_2}{\normm{x_2}^2}.
     $
\end{itemize}
To simply notations, let $v_{\alpha}'(x_1,x_2) := v_{\alpha}(x_1,x_2) + v_{\alpha}(x_1,-x_2)$ and
 $v_{\alpha}''(x_1,x_2):= \log( u_{\alpha}(x_1,x_2) )v_{\alpha}(x_1,x_2) + \log( u_{\alpha}(x_1,-x_2) ) v_{\alpha}(x_1,-x_2)$. Then
\begin{equation*}
\begin{aligned}
      Z^{\alpha}_{\log,\nu_i}(\tau) = &\,
       \int_{-1}^1 \int_{0}^{1} \log \left( {x_2}^{-1} \right) {x_2}^{\tau-2}  v_{\alpha}'(x_1,x_2){\diff x_2 \diff x_1}
       +
         \int_{-1}^1 \int_{0}^{1} v_{\alpha}''(x_1,x_2) {x_2}^{\tau-2} {\diff x_2 \diff x_1}
      \\
         =&: 
         A^{\alpha}(\tau) + B^{\alpha}(\tau).
\end{aligned}
\end{equation*}
We now simplify the height zeta functions using integration by part. We deal with $A^{\alpha}$ first and then $B^{\alpha}$. 
\begin{equation*}
\begin{aligned}
   A^{\alpha}(\tau)  =&\,
     \int_{-1}^1 \int_{0}^{1} -  \frac{\partial}{\partial x_2} \left( v_{\alpha}'(x_1,x_2)\log \left( {x_2}^{-1} \right) \right) \frac{ {x_2}^{\tau-1}}{\tau - 1}   {\diff x_2 \diff x_1}
     \\     
     =&\,
     \frac{ 1}{\tau - 1}  \int_{-1}^1 \int_{0}^{1}   \frac{ \partial  v_{\alpha}'(x_1,x_2)  }{\partial x_2}  \log \left( x_2 \right) {x_2}^{\tau-1}   \diff x_2 \diff x_1
     +
     \int_{-1}^1 \int_{0}^{1}  v_{\alpha}'(x_1,x_2) {x_2}^{-1}   \frac{ {x_2}^{\tau-1}}{\tau - 1}   {\diff x_2 \diff x_1}
     \\     
     = &\,
     \frac{ 1}{\tau - 1}  \int_{-1}^1 \int_{0}^{1}  \frac{ \partial  v_{\alpha}'(x_1,x_2)  }{\partial x_2}  \log \left( {x_2} \right) {x_2}^{\tau-1} {\diff x_2 \diff x_1}
     +
     \frac{1}{(\tau - 1)^2} 
     \int_{-1}^1 \int_{0}^{1}  -  \frac{ \partial  v_{\alpha}'(x_1,x_2)  }{\partial x_2}  {x_2}^{\tau-1}  {\diff x_2 \diff x_1}
     \\
     =&:\,
     \frac{ A^{\alpha}_1(\tau)}{\tau - 1}+
     \frac{A^{\alpha}_2(\tau)}{(\tau - 1)^2}  .
\end{aligned}
\end{equation*}
Note that $A^{\alpha}_1(\tau)$, $A^{\alpha}_2(\tau)$ and their derivatives are analytic when $\tau >0$, bounded in vertical strips, and 
\[
  A^{\alpha}_2(1)  = \int_{-1}^1  v_{\alpha}'(x_1,0) \diff x_1 =  2 \int_{-1}^1  v_{\alpha}(x_1,0) \diff x_1 >0.
\]
Similarly,
\begin{equation*}
\begin{aligned}
            B^{\alpha}(\tau) =
           \frac{1}{{ \tau -1 } }   \int_{-1}^1 \int_{0}^{1} - \frac{ \partial  v_{\alpha}''(x_1,x_2)  }{\partial x_2}  {x_2^{\tau -1}}{\diff x_2 \diff x_1}
           = \frac{B^{\alpha}_1(\tau)}{ \tau -1} 
\end{aligned}
\end{equation*}
with $B^{\alpha}_1(\tau)$, as well as its derivatives, being analytic and bounded in vertical strips on the half space $\mathrm{Re}(\tau) >0$.
So we conclude that 
$Z_{\log,\nu_i}(\tau)$ enjoys the following property
\begin{itemize}
     \item  it has a meromorphic continuation to the half space $\mathrm{Re}(\tau)> 0$;
     \item  it has  a pole of order $2$  at $\tau=1$ with positive residue and is analytic elsewhere;
     \item its derivatives have moderate growth in vertical strips on the half space $\mathrm{Re}(\tau)> 0$.
\end{itemize}
  For $Z^{\alpha}_{\nu_i}(\tau) $,
\begin{equation*}
\begin{aligned}
           Z_{\nu_i}(\tau)  =
           \frac{1}{{ \tau -1 } }   \int_{-1}^1 \int_{0}^{1} - \frac{ \partial  v_{\alpha}'(x_1,x_2)  }{\partial x_2}  {x_2^{\tau -1}}{\diff x_2 \diff x_1}
           = :\frac{C^{\alpha}_1(\tau)}{ \tau -1} 
\end{aligned}
\end{equation*}
with $C^{\alpha}_1(\tau)$ analytic on $\mathrm{Re}(\tau)>0$ and
\[
    C^{\alpha}_1(1) =  \int_{-1}^1 \int_{0}^{1} - \frac{ \partial  v_{\alpha}'(x_1,x_2)  }{\partial x_2}   \diff x_2 \diff x_1
    =  \int_{-1}^1 v_{\alpha}'(x_1,0) \diff x_1 >0.
\] 
Thus $Z_{\nu_i}(\tau)$ has the same properties as $Z_{\log,\nu_i}(\tau)$ except that the order of pole is $1$ instead of $2$.

Now we briefly indicate the needed modification when the place $\nu_j$ is complex.
\begin{itemize}
    \item $(\calU_{\alpha})$ is an open covering of $\overbmX(\C)$ and $\phi_{\alpha}: \overbmX(\C) \to [0,1]$ is a smooth function supported on $\calU_{\alpha}$.
    \item $\Psi_{\alpha}: \big( (-1,1)+i(-1,1) \big)^2 \cong \calU_{\alpha}$, $\Psi_{\alpha}^{-1} \left(\calU_{\alpha}\cap \bmD(\C) \right)= \{(z_1,z_2) ,\; z_2 =0\}$, and $z_2 \circ \Psi^{-1}_{\alpha}$ is a local generator of $\bmD$.
    \item $u_{\alpha},v_{\alpha} :  \big( (-1,1)+i(-1,1) \big)^2 \to [0,+\infty)$ with similar properties as before.
    \item The zeta functions look slightly different
    \[
     Z^{\alpha}_{\log,\nu_j}(\tau) =  \int_{-1}^1 \int_{-1}^{1} \int_{-1}^1 \int_{-1}^{1} 
      \log \left( \normm{x_2+iy_2}^{-2} u_{\alpha} (z_1,z_2)^2 \right) \normm{x_2+iy_2}^{2 \tau} v_{\alpha} (z_1,z_2)\, \frac{\diff x_2 \diff y_2 \diff x_1 \diff y_1}{\normm{x_2+iy_2}^4}.
     \]
     Using polar coordinate $z_2 = r_2 e^{i \theta_2}$, we have 
      \begin{equation*}
     \begin{aligned}
            Z^{\alpha}_{\log,\nu_j}(\tau)  &=  \int_{-1}^1 \int_{-1}^{1} \int_{0}^{2\pi} \int_{0}^{\infty} 
              \log \left( r_2^{-2} u_{\alpha}(z_1,z_2)^2 \right) r_2^{2 \tau} v_{\alpha} (z_1, z_2)\, \frac{r_2 \diff r_2 \diff \theta_2 \diff x_1 \diff y_1}{r_2^4}
              \\
              &=
               \int_{-1}^1 \int_{-1}^{1} \int_{0}^{2\pi} \int_{0}^{\infty} 
              \log \left( r_2^{-2} u_{\alpha} (z_1,z_2)^2 \right) r_2^{2 \tau-3} v_{\alpha} (z_1, z_2) \,{ \diff r_2 \diff \theta_2 \diff x_1 \diff y_1}.
     \end{aligned}
      \end{equation*}
      \item Likewise, the other zeta function is
      \[
     Z^{\alpha}_{\nu_j}(\tau) =   \int_{-1}^1 \int_{-1}^{1} \int_{0}^{2\pi} \int_{0}^{\infty} 
               r_2^{2 \tau-3} v_{\alpha} (z_1, z_2) \, { \diff r_2 \diff \theta_2 \diff x_1 \diff y_1}.
     \]
\end{itemize}
By performing integration by parts as before, one finds that $Z_{\log,\nu_j}(\tau)$
\begin{itemize}
     \item  has a meromorphic continuation to the half space $\mathrm{Re}(\tau)> 0.5$;
     \item  has  a pole of order $2$  at $\tau=1$ with positive residue and is analytic elsewhere;
     \item its derivatives have moderate growth in vertical strips on the half space $\mathrm{Re}(\tau)> 0.5$.
\end{itemize}
Again, the same thing holds for $Z_{\nu_j}(\tau)$ except that the order of pole is $1$ instead of $2$.

The proof is complete by using Eq.(\ref{equation_height_zeta_product_places}).
\end{proof}

\section{Effective equidistribution of lines on tori}\label{appendix_B}

We prove Lemma \ref{lemma_equidistribution_coordinate_lines_tori_number_fields}. Let us state it again for convenience.
\begin{lem}\label{lemma_equidistribution_coordinate_lines_tori_number_fields_appendix}
Let $\Lambda$ be a lattice in $\R^l$ commensurable with the geometric embedding of $\calO_k$. 
Take the normalization $\int_{\R^l/\Lambda} \diff[\vec{v}] =1$.
There exists some $0<\delta <0.5$ such that for  every $T_0\in \R$ and $T$ sufficiently large, the following holds for every Lipschitz-continuous function $f$:
\begin{itemize}
    \item[(1)] For $i=1,...,l_1$, 
    \begin{equation*}
    \begin{aligned}
                  \frac{1}{T} \int_{T_0}^{T_0+T} f([t \vec{e}_i  ]) \, \diff t = 
                  \int_{\R^l/\Lambda} f([\vec{v}]) \, \diff [\vec{v}] + \bmO \Big(\frac{\Lip(f)}{T^{\delta}} \Big).
   \end{aligned}
    \end{equation*}
    \item[(2)] For $j=1,..., l_2$ and every interval $I$ of length at most $2\pi$,
     \begin{equation*}
    \begin{aligned}
                  \frac{1}{T} \int_{T_0}^{T_0+T}  \int_{I}
                  f([t e^{i\theta}  \vec{e}_{l_1+j}  ]) \,\diff \theta \diff t = 
                  |I| \int_{\R^l/\Lambda} f([\vec{v}]) \, \diff [\vec{v}] + \bmO \Big(\frac{\Lip(f)}{T^{\delta}} \Big).
   \end{aligned}
    \end{equation*}
\end{itemize}
\end{lem}

We shall need the following.
\begin{lem}\label{lemma_Appendix_irrational_lines}
Let $\Lambda \leq \R^l$ be a lattice with $l\geq 1$. There exists  $\kappa >4$ such that 
whenever $(\delta, \vec{\alpha}, N)$ satisfies
\begin{itemize}
   \item[(1)] $0 < \delta < 0.5 $, $\vec{\alpha} \in \R^l$, $N>1$;
   \item[(2)] $\displaystyle \normm {  \scrL(\vec{\alpha})   } \geq \frac{\delta^{-\kappa (2l+2)}}{ N} \, \text{ for all }  \, \scrL_{\neq 0} \in \mathrm{Hom}(\Lambda, \Z) \subset (\R^l)^* ,\; \norm{\scrL}\leq \delta^{-\kappa(2l+2)}$,
\end{itemize} 
then for every Lipschitz-continuous function $f: \R^l / \Lambda \to \R$ and $N_0\in \R$,
\begin{equation*}
\begin{aligned}
    \frac{1}{N} \int_{N_0}^{N_0+N} f ([t\vec{\alpha}]) \, \diff t = \int_{\R^l/\Lambda }  f([\vec{v}]) \, \diff[\vec{v}] + \bmO(\Lip(f) \delta )
\end{aligned}
\end{equation*}
where it is normalized such that $\int_{\R^l/\Lambda } \diff[\vec{v}] =1 $.
\end{lem}

This will be deduced from an effective version of  Kronecker's equidistribution theorem, which can be found in \cite[Proposition 3.1]{Green_Tao_nilmanifold}.
\begin{prop}\label{proposition_effective_Kronecker}
For $l \in \Z^+$, there exists $\kappa>1$ such that if $0<\delta<0.5$, $\vec{\alpha} \in \R^l$ and $N\in \Z^+$ satisfy
\begin{equation}\label{equation_condition_kronecker}
\begin{aligned}
   \normm{\vec{v} \cdot \vec{\alpha} - L} \geq \frac{\delta^{-\kappa}}{N},\quad
   \text{for all } \; 
   \vec{v} \in \Z^l,\;
   L\in \Z,\; 0<\vec{v} \leq \delta^{-\kappa},
\end{aligned}
\end{equation}
then for every Lipschitz-continuous function $f: \R^l /\Lambda \to \R$,
\begin{equation*}
\begin{aligned}
     \normm{
      \frac{1}{N} \sum_{k=1}^N f ( [k \vec{\alpha}]) - \int_{\R^l/\Lambda } f ([\vec{v}]) \, \diff[\vec{v}]
     }
     \leq \delta \Lip( f ).
\end{aligned}
\end{equation*}
\end{prop}

\begin{proof}
     Without loss of generality we assume $\Lambda=\Z^l$ and $N_0=0$. 
     Take a triple $(\delta, \vec{\alpha}, N)$ satisfying the condition listed in Lemma \ref{lemma_Appendix_irrational_lines}. We may further assume $\norm{\vec{\alpha}}\leq 1$ by replacing $\vec{\alpha}$ by $\vec{\alpha}/M$ and $N$ by $NM$ for a positive integer $M$.
     
     We will firstly show that for some $\lambda\in (-0.5,0)$, the triple $(\delta, (1+\lambda)\vec{\alpha} , N')$, for any integer larger than $N$, satisfies Eq.(\ref{equation_condition_kronecker}).
     To achieve this, we define for each pair $(\vec{v}, L) \in \Z^l \times \Z$ with $\norm{\vec{v}} \leq \delta^{-\kappa}$ (hence is also bounded by $\delta^{-\kappa(2l+2)}$),
     \begin{equation*}
\begin{aligned}
      \mathrm{BAD}(\vec{v},L)  := \left\{
         \lambda \in (-0.5,0) \midd
         \normm{   \vec{v}\cdot \vec{\alpha} - L + \lambda  \vec{v}\cdot \vec{\alpha}
         } 
         \leq \frac{\delta^{-\kappa}}{N}
      \right\}.
\end{aligned}
\end{equation*}
     By assumption, the ``slope'' $\normm{\vec{v}\cdot \vec{\alpha}} \geq N^{-1} {\delta}^{-\kappa (2l+2)}$, we have 
     \begin{equation*}
\begin{aligned}
   \normm{  \mathrm{BAD}(\vec{v},L)  } \leq \frac{ 2N^{-1}\delta^{-\kappa}
   }{    N^{-1} {\delta}^{-\kappa (2l+2)}
   } = 2\delta^{\kappa(2l+1)}.
\end{aligned}
\end{equation*}
     Let
     \begin{equation*}
\begin{aligned}
     \mathrm{BAD}:= \bigcup_{   0 < \norm{\vec{v}}   \leq \delta^{-\kappa} 
     } \mathrm{BAD}(\vec{v},L)  .
\end{aligned}
\end{equation*}
Then for $\lambda \in (-0.5,0)\setminus \mathrm{BAD}$, the triple  $(\delta, (1+\lambda)\vec{\alpha} , N)$ satisfies Eq.(\ref{equation_condition_kronecker}).
To find such a $\lambda$, we must show $\normm{\mathrm{BAD}} <0.5$.

Note that a bound has been implicitly imposed on $L$:
\begin{equation*}
\begin{aligned}
   0 < \norm{\vec{v}}   \leq \delta^{-\kappa},\; \mathrm{BAD}(\vec{v},L)  \neq \emptyset
   \implies 
   \normm{L} \leq \normm{\vec{v}} \normm{\vec{\alpha}} + N^{-1} \delta^{-\kappa} \leq 2 N^{-1} \delta^{-\kappa}.
\end{aligned}
\end{equation*}
Therefore,
 \begin{equation*}
\begin{aligned}
        &  \mathrm{BAD}:= \bigcup_{  \substack{
           0 < \norm{\vec{v}}   \leq \delta^{-\kappa} ,\\ \normm{L} \leq 2N^{-1}\delta^{-\kappa}
           }
     } \mathrm{BAD}(\vec{v},L) 
     \implies
     &
     \normm{\mathrm{BAD}}
     \leq 3^l \delta^{- \kappa   l} \cdot 2N^{-1} \delta^{-\kappa} \cdot 2\delta^{\kappa(2l+1)} < 2(3\delta^{\kappa})^l.
\end{aligned}
\end{equation*}   
But $\delta <0.5$ and $\kappa >4$ imply that $\delta^{\kappa} < 1/16$ and hence $ 2(3\delta^{\kappa})^l <0.5$.

Now  fix such a $\lambda \in (-0.5,0)\setminus \mathrm{BAD}$ and write 
\[
\vec{\alpha}_{\lambda}:= (1+\lambda)\vec{\alpha },\quad N_{\lambda}:= \left \lfloor \frac{N}{1+\lambda} \right\rfloor.
\]
Thus  $(\delta, \vec{\alpha}_{\lambda}, N_{\lambda})$ satisfies Eq.(\ref{equation_condition_kronecker}) and Proposition \ref{proposition_effective_Kronecker} is ready to be applied.
\begin{equation*}
\begin{aligned}
     \frac{1}{N} \int_0^N f([t.\vec{\alpha}]) \diff t
     = &
      \frac{1+\lambda }{N} \int_{0}^{\frac{N}{1+\lambda}} f ([t \vec{\alpha}_{\lambda}]) \diff t
     \\
     =
     &
       \frac{1+\lambda }{N} \int_0^1 \sum_{k=0}^{N_{\lambda}-1 } f ([s  \vec{\alpha}_{\lambda}+ k \vec{\alpha}_{\lambda}]) \,  \diff s
       + \bmO \Big(\frac{1+\lambda}{N} \Lip{f}  \Big)
       \\
       =&
       \int_0^1
       \frac{N_{\lambda}}{ N/(1+\lambda) } \int_{\R^l/\Z^l} f([\vec{v}]) \,\diff [\vec{v}] \,
       \diff s
        + \bmO(\delta \Lip(f)) + \bmO\Big(\frac{1+\lambda}{N} \Lip(f) \Big)
        \\
        =&
         \int_{\R^l/\Z^l} f([\vec{v}]) \,\diff [\vec{v}] + \bmO \Big( \Lip(f)(\delta + \frac{2}{N})  \Big).
\end{aligned}
\end{equation*}
To kill the additional $\frac{2}{N}$, we apply the above equation to $\vec{\alpha}$ replaced by $\vec{\alpha}/M$ and $N$ replaced by $NM$ for a positive integer $M$. It can be seen from the definition that the new triple $(\delta, \vec{\alpha}/M, NM )$ still verifies the assumption. Therefore,
\begin{equation*}
\begin{aligned}
     \frac{1}{MN} \int_0^N f([\frac{t}{M}\vec{\alpha}]) \diff t =   \int_{\R^l/\Z^l} f([\vec{v}]) \,\diff [\vec{v}] + \bmO\Big( \Lip(f)(\delta + \frac{2}{MN}) \Big).
\end{aligned}
\end{equation*}
But by a change of variable, the left hand side is equal to 
\[
       \frac{1}{N} \int_0^N f([{t}\vec{\alpha}]) \diff t.
\]
So letting $M \to +\infty$ completes the proof.
\end{proof}

\begin{proof}[Proof of Lemma \ref{lemma_equidistribution_coordinate_lines_tori_number_fields_appendix}]
For simplicity, we will only prove the second part. And without loss of generality take $T_0=0$ and $I \subset [0,2\pi)$.

To avoid confusion, we distinguish  $\beta \in k$ from its geometric embedding in $k_{\infty}$, denoted as $\vec{v}_{\beta}$. 
Associating $\beta \in \calO_k$ the map $\scrL_{\beta}(\alpha) := \tr_{k/\Q}(\alpha \cdot \beta) $ embeds $\calO_k$ into $\Hom_{\Z}(\calO_k,\Z)$ as a finite-index subgroup, which naturally sits inside $\Hom_{\R}(k_{\infty},\R)$. 
Since $\Lambda$ is commensurable with $\calO_k$, $\Hom_{\Z}(\Lambda,\Z)$ contains a finite index subgroup where all the elements come from $\calO_k$. Let $L_0$ be this index.

For $\vec{v}=(v_1,...,v_{l_1+l_2}) \in k_{\infty}$, 
\[
     \scrL_{\beta}(\vec{v}) = \sum_{i=1}^{l_1} \sigma_i(\beta) v_i + \sum_{j=1}^{l_2} 2 \mathrm{Re}(\sigma_{l_1+j}(\beta) v_{l_1+j}).
\]
In particular,
\begin{equation*}
\begin{aligned}
    \scrL_{\beta}(\vec{e_i}) &= \sigma_i(\beta) ,\quad i=1,..., l_1;
    \\
    \scrL_{\beta}(\vec{e}_{l_1+j}  ) &=  2\mathrm{Re}(\sigma_{l_1+j}(\beta)) ,\;\;
     \scrL_{\beta}( i \vec{e}_{l_1+j}  ) = - 2\mathrm{Im}(\sigma_{l_1+j}(\beta)),
    \quad j=1,..., l_1.
\end{aligned}
\end{equation*}
Its norm is 
\begin{equation*}
\begin{aligned}
    \norm{  \scrL_{\beta} } = \left(
       \sum_{i=1}^{l_1} \normm{\sigma_i(\beta)}^2 +4 \sum_{j=1}^{l_2} \normm{\sigma_{l_1+j} (\beta)}^2
    \right)^{\frac{1}{2}} ,
\end{aligned}
\end{equation*}
larger than $\normm{\sigma_{j}(\beta)}$  for every $j=1,...,l_1+l_2$.
Since $\Nm_{k/\Q}(\beta) \geq 1$ for nonzero $\beta \in \calO_k$, we have
\begin{equation}\label{equation_lower_bound_algebraic_integer}
\begin{aligned}
     \normm{\sigma_{l_1+j}(\beta)} \geq \frac{1}{  \norm{\scrL_{\beta}}  ^{l-1} }.
\end{aligned}
\end{equation}

Define $\delta_1,\delta_2 \in (0,0.5)$ by
\begin{equation}\label{equation_define_delta_1_delta_2}
\begin{aligned}
    \delta_1:= 0.2,\quad 
    \delta_2 := 0.1\kappa^{-1} (2l+2)^{-2}.
\end{aligned}
\end{equation}

For $\beta_{\neq 0} \in \calO_k$, consider the set
\[
   E(\beta):= \left\{
       \theta \in (0,2\pi] \midd
       \normm{ \mathrm{Re}(\sigma_{l_1+j} (\beta)  e^{i\theta } )  } 
       \leq T^{-\delta_1} \normm{ \sigma_{l_1+j} (\beta)  } 
   \right\}.
\]
Then $\normm{E(\beta)} \leq \frac{2\pi}{T^{\delta_1}}$ and for $\theta \in [0,2\pi) \setminus E(\beta)$,
\begin{equation}\label{equation_estimate_L_beta}
\begin{aligned}
     \normm{ \scrL_{\beta}( e^{i\theta} \vec{e}_{l_1+j} )  } =  
     2 \normm{ \mathrm{Re}(\sigma_{l_1+j} (\beta)  e^{i\theta } )  } 
     \geq \frac{2}{T^{\delta_1}  \norm{\scrL_{\beta}}  ^{l-1} }
\end{aligned}
\end{equation}
where we have used Eq.(\ref{equation_lower_bound_algebraic_integer}).
Let us fix another constant $L_1>0$ such that 
\[
    \# \left\{ \beta\in \calO_k \midd 
       \norm{ \scrL_{\beta}  } \leq N
    \right\} \leq L_1 N^{ l_1+l_2 }, \quad \forall\, N >1.
\]
Also let
\[
   E:= \bigcup_{  \norm{\scrL_{\beta}}   \leq L_0 T^{\delta_2 \kappa (2l+2)}
    } E(\beta).
\]
Then by definition of $L_1$ and Eq.(\ref{equation_define_delta_1_delta_2}),
\begin{equation}\label{equation_estimate_E}
   \normm{E} \leq \frac{2\pi  L_1 L_0^{l_1+l_2 } T^{(l_1+l_2) \delta_2 \kappa (2l+2)  }
   }{ T^{\delta_1} } \leq \frac{2\pi L_1 L_0^{l_1+l_2}
   }{  T^{0.1}  }.
\end{equation}

For $\theta \in  [0,2\pi) \setminus E $ and $T$ sufficiently large,
 we hope to verify the triple $( T^{-\delta_2} ,  e^{i\theta} \vec{e}_{l_1+j}  ,  T )$ (was denoted as $(\delta, \vec{\alpha}, N)$ in Lemma  \ref{lemma_Appendix_irrational_lines})
satisfies the condition listed in Lemma  \ref{lemma_Appendix_irrational_lines}, that is to say,
\begin{equation*}
\begin{aligned}
          \normm{  \scrL(e^{i \theta }  \vec{e}_{l_1+j} ) } \geq  \frac{   T^{\delta_2 \kappa (2l+2)}
          }{  T  },\quad \text{for all } \scrL_{\neq 0} \in \Hom_{\Z}(\Lambda,\Z) \; \text{ satisfying } \norm{\scrL} \leq T^{\delta_2 \kappa (2l+2) }.
\end{aligned}
\end{equation*}
Take such an $\scrL$, by assumption, $L_0\scrL = \scrL_{\beta}$ for some $\beta \in \calO_k$. Applying Eq.(\ref{equation_estimate_L_beta}), we get
\begin{equation*}
\begin{aligned}
       \normm{\scrL (e^{i\theta} \vec{e}_{l_1+j})  } = &
       \normm{ L_0 \scrL_{\beta}(e^{i\theta} \vec{e}_{l_1+j})  } 
       \\
       \geq &
       \frac{2L_0}{T^{\delta_1}  \norm{\scrL_{\beta}}^{l-1}  }
       \\
       \geq &
       \frac{2}{L_0^{\delta_2 \kappa (2l+2) -1
       }   }
       \cdot 
       \frac{  T^{1- \delta_1 - \delta_2 \kappa (2l+2)}
       }{ T }.
\end{aligned}
\end{equation*}
By Eq.(\ref{equation_define_delta_1_delta_2}),  $1- \delta_1 -\delta_2\kappa(2l+2) \geq \delta_2 \kappa(2l+2) + 0.1$.
Hence 
\begin{equation*}
\begin{aligned}
     \normm{\scrL (e^{i\theta} \vec{e}_{l_1+j})  } \geq  
      \frac{2 T^{0.1}}{L_0^{\delta_2 \kappa (2l+2) -1
       }   }
       \cdot 
       \frac{  T^{\delta_2 \kappa (2l+2)}
       }{ T } \geq 
       \frac{  T^{\delta_2 \kappa (2l+2)}
       }{ T } 
\end{aligned}
\end{equation*}
when $T$ is sufficiently large. Now Lemma \ref{lemma_Appendix_irrational_lines} is ready to apply for such $\theta$'s.

Let $I_1:= I \cap E$ and $I_2 := I \setminus E$. 
Thus by Eq.(\ref{equation_estimate_E}), $\normm{I_1} \leq  \frac{2\pi L_1 L_0^{l_1+l_2}
   }{  T^{0.1}  }$.
On applying Lemma \ref{lemma_Appendix_irrational_lines}, we find that 
\begin{equation*}
\begin{aligned}
    \frac{1}{T} \int_0^T \int_{I_2} f( [t e^{i \theta} \vec{e}_{l_1+j } ]) \, \diff \theta \diff t
    & =  \normm{ I_2 } \int_{\R^l/\Lambda} f([\vec{v}]) \,\diff [\vec{v}] 
    + \bmO\big(   \Lip(f)  T^{-\delta_2}
    \big)
    \\
    &=
    \normm{ I } \int_{\R^l/\Lambda} f([\vec{v}]) \,\diff [\vec{v}] 
    + \bmO\big( (1 + 2\pi L_1 L_0^{l_1+l_2}  )   \Lip(f)   T^{-\delta_2}
    \big).
\end{aligned}
\end{equation*}
On the other hand,
\begin{equation*}
\begin{aligned}
     \frac{1}{T} \int_0^T \int_{I_1} f( [t e^{i \theta} \vec{e}_{l_1+j}   ]) \, \diff \theta \diff t
     \leq \norm{f }_{\sup} \normm{I_1} 
     \leq  \Lip(f ) \frac{   2\pi L_1 L_0^{l_1+l_2}
       }{  T^{0.1}  }.
\end{aligned}
\end{equation*}
Combining the above two equations,
the proof of part (2) of the Lemma \ref{lemma_equidistribution_coordinate_lines_tori_number_fields} is complete with $\delta:= 0.5 \delta_2$ and sufficiently large $T$ (explicitly depending on $\delta_2, L_0,L_1$).

\end{proof}

\section{Meromorphic continuation of Eisenstein series and counting}\label{appendix_C}

In this appendix take $\Gamma:= \SL_2(\calO_{k})$. This restriction, which is not essential, comes from the reference \cite{Jorgenson_Lang_Asai}.
The purpose of this appendix is to show 
\begin{prop}
There exist $K_0$-invariant continuous functions $c^{+}$ and $c^{-}$ on $Y$ such that 
\[
   \scrD^{\star}(\varphi) = \int_Y \varphi(y) c^{\star}(y) \, \diff \rmm_Y(y)
\]
for $\star= +,-$ where 
$\scrD^{+}(\varphi)$ and $\scrD^{-}(\varphi)$ are as in Theorem \ref{theorem_equidistribution_infinite_homogeneous_quadric}. Consequently, the $c_2(\varphi)$ in Theorem \ref{theorem_smoothed_count_quadric} tends to a constant as the support of $\varphi$ shrinks to some singleton $y_0\in Y$.
\end{prop}
For the sake of simplicity, we only deal with $\scrD^+$.

\subsection{Siegel-like transforms}
Consider
\[
\begin{tikzcd}
& K_0 \bs G / H^{(1)}U \cap \Gamma 
 \arrow[dl] \arrow[dr] & \\
 K_0 \bs G /  \Gamma 
& &  K_0 \bs G / H^{(1)}U  \cong \R^+
\end{tikzcd}
\]
where the last isomorphism is provided by $x\in \R^+$ mapped to the coset containing $h^{\Delta}_{x^{-1/l}}$.
Conversely, given $K_0g H^{(1)}U$, the corresponding $x\in \R^+$ can be recovered as $x= \norm{g.e_1}_{k_{\infty}}^{-1}$.
We fix some $c_{G}^+>0$ such that under the surjection, 
\[
   K_0 \times \R^+ \times H^{(1)}U \to G,\quad
   (k,x,u)\mapsto kh^{\Delta}_{x^{-1/l}}u
\]
the measure $\rmm_{K_0}\otimes l^{-1}x^{-2} \frac{\diff x}{x} \otimes \rmm_{H^{(1)}U}$ is sent to $c_{G}^+ \rmm_G$.
Here $\rmm_{K_0}$ is the probability Haar measure and $\rmm_{H^{(1)}U}$ and $\rmm_G$ are chosen such that the corresponding measure on the quotient space $G/\Gamma$ is a probability measure.
The factor $x^{-2}$ appears because the modular character for $ h^{\Delta}_{x^{-1/l} } $ on $H^{(1)}U$ is $x^{-2}$.
With this convention in mind, let $\calB_{\geq 0}(X)$  be the set of non-negative Borel measureable functions on some topological space $X$.
Define  Siegel-like transforms $\calF: \calB_{\geq 0}(K_0 \bs G /  \Gamma) \to \calB_{\geq 0}(\R^+)$ and $\calG: \calB_{\geq 0}(\R^+) \to \calB_{\geq 0}(K_0 \bs G /  \Gamma)$ via
\begin{equation*}
\begin{aligned}
      \calF(\varphi)(x)
      &:= \int\varphi( h^{\Delta}_{x^{-1/l} } .z) \, \diff \rmm^U_{[H^{(1)}]}(z),
      \\
      \calG(f)([g])
      & := \sum_{\gamma \in \Gamma/\Gamma \cap H^{(1)}U } f ( \norm{g\gamma.e_1}^{-1}_{k_{\infty}}  ).
\end{aligned}
\end{equation*}
Then 
\[
     \int \calF(\varphi)(x)  f(x)  \, \frac{1}{l}x^{-2} \frac{\diff x}{x} = c_G^+ \int \varphi([g]) \calG(f)([g])\, \diff \rmm_Y(y).
\]
With these notations, rewrite $\scrD^+(\varphi)$ as (so from now on, we assume $\varphi\geq 0$ smooth, compactly supported):
\begin{equation*}
\begin{aligned}
      \scrD^+(\varphi)
      =& \frac{1}{l} \int_0^{\infty} \calF(\varphi)'(x)(-\log {x}) \, \diff x
      \\
      =&
      \frac{1}{l} \int_0^{1} \big( \calF(\varphi)(x)- \int \varphi \, \diff \rmm_Y \big)' (-\log {x}) \, \diff x
      +
      \frac{1}{l} \int_1^{\infty} \calF(\varphi)'(x)(-\log {x}) \, \diff x
      \\
      =&
       \frac{1}{l} \int_0^{1} \big( \calF(\varphi)(x) - \int \varphi  \big) \, \diff \rmm_Y   \, x^{-1} \diff x
       + \frac{1}{l} \int_1^{\infty} \calF(\varphi)(x) \, x^{-1} \diff x.
\end{aligned}
\end{equation*}
Now we introduce an additional parameter $s>0$,
\begin{equation*}
\begin{aligned}
      \scrD^+(\varphi)
      =& 
      \lim_{s\to 0^+} \frac{1}{l} \int_0^{1}  \big( \calF(\varphi)(x) - \int \varphi  \, \diff \rmm_Y \big) 
       \, x^{s-1} \diff x
       + \frac{1}{l} \int_1^{\infty} \calF(\varphi)(x) \, x^{s-1} \diff x
      \\
      =& 
      \lim_{s\to 0^+}
      -\frac{\int \varphi \, \diff\rmm_Y}{ls }
      +
        \int_0^{\infty}  \calF(\varphi)(x)  x^{s+2}\, \frac{1}{l}x^{-2} \frac{\diff x}{x}
        \\
        =&
         \lim_{s\to 0^+}
      -\frac{\int \varphi \, \diff\rmm_Y}{ ls }
      +
       c_G^+ \int_Y  \varphi(y) \calG( x^{s+2} )(y)\, \diff \rmm_Y(y).
\end{aligned}
\end{equation*}
Before showing that $ \calG( x^{s+2} )$ is an Eisenstein series, we firstly review some preliminaries from \cite[Section 3]{Jorgenson_Lang_Asai}.

\subsection{The pole of primitive Eisenstein series}

\subsubsection{Eisenstein series}
Let us recall some .
Let 
\[
    \bmh_{\R}:= \left\{ x+iy \midd x\in \R ,y \in \R^+\right\}
    ,
    \quad
    \bmh_{\C}:= \left\{ z=
    \begin{bmatrix}
    x & -y \\
    y & \overline{x}
    \end{bmatrix}
     \midd x\in \C ,y \in \R^+\right\}.
\]
Then $\SL_2(\R)$ (resp. $\SL_2(\C)$) acts on $\bmh_{\R}$ (resp. $\bmh_{\C}$).
Let $\bmh_{k_{\infty}}:= \prod_{i=1,...,l_1+l_2} \bmh_{k_{\nu_i}}$. Then $G=\SL_2(k_{\infty})$ acts on $\bmh_{k_{\infty}}$.
Given $\vec{z}=(z_i)\in \bmh_{k_{\infty}}$, define
\begin{equation}\label{definition_Ny}
    \bmN y( \vec{z} ):= \prod_{i=1}^{l_1+l_2}  y(z_i)^{\epsilon_i}
\end{equation}
with $\epsilon_i=1$ if $k_{\nu_i}\cong \R$ and $\epsilon_i =2$ if $k_{\nu_i}\cong \C$.

For a pair $(\alpha,\beta)\in k^2 \setminus \{(0,0)\}$, let 
\[
    \{ \alpha, \beta \}:= \left\{
     (\alpha', \beta') \midd (u\alpha, u\beta)= (\alpha', \beta') ,\; \exists \, u \, \in \calO_k^{\times}
    \right\}
\]
Given $(\alpha,\beta)\in k^2 \setminus \{(0,0)\}$, define
\[
    \bmN y (\alpha, \beta; \vec{z}):= \prod_{i=1}^{l_1+l_2} \frac{y(z_i)}{
    \normm{\sigma_i(\alpha) x(z_i)+ \beta} + \normm{\sigma_i(\alpha)}^2 y(z_i)^2
    },
\]
which depends only on the equivalence class $\{\alpha,\beta\}$.
For a fractional ideal $\fraka$ of $\calO_k$, let 
\begin{equation*}
\begin{aligned}
   \mathrm{Equ}(\fraka)
   &:= \left\{ \{\alpha, \beta\}
   \midd
   \calO_k \alpha + \calO_k \beta \subset \fraka
   \right\}
   \\
   \mathrm{Equ}^*(\fraka)
   &:= \left\{ \{\alpha, \beta\}
   \midd
   \calO_k \alpha + \calO_k \beta = \fraka
   \right\}
\end{aligned}
\end{equation*}
and
\begin{equation*}
\begin{aligned}
       E(\vec{z},s,\fraka) &:= \sum_{\{\alpha,\beta\} \in \mathrm{Equ}(\fraka)} 
       \bmN y (\alpha, \beta; \vec{z})^s \Nm(\fraka)^{2s}
       ,
       \\
        E^*(\vec{z},s,\fraka) &:= \sum_{\{\alpha,\beta\} \in \mathrm{Equ}^*(\fraka)} 
       \bmN y (\alpha, \beta; \vec{z})^s \Nm(\fraka)^{2s}.
\end{aligned}
\end{equation*}
It follows from the definition that $ E(\vec{z},s,\fraka)$ and $ E^*(\vec{z},s,\fraka)$ only depend on the ideal class containing $\fraka$.
Let $\mathrm{Cl}(\calO_k)$ be the ideal class group of $\calO_k$.
In \cite[Section 3, Equation (2)]{Jorgenson_Lang_Asai}, for $\fraka\in \frakK\in \mathrm{Cl}(\calO_k)$, it is defined that
\begin{equation*}
\begin{aligned}
       E(\vec{z},s,\mathfrak{K}) := E(\vec{z},s,\fraka^{-1}),
       \quad 
       E^*(\vec{z},s,\frakK) := E^*(\vec{z},s,\fraka^{-1}).
\end{aligned}
\end{equation*}
The meromorphic continuation property of $E(\vec{z},s,\mathfrak{K})$ is known (see \cite[Theorem 3.3]{Jorgenson_Lang_Asai}):
\begin{thm}\label{theorem_Eisenstein_series}
Let $\Gamma(s)$ be the standard Gamma function and define 
\[
    G_k(2s):= \mathrm{disc}(\calO_k) \pi^{-sl} \Gamma(s)^{l_1}\Gamma(2s)^{l_2}.
\]
Then there exists a continuous function $c_{E,\frakK}: K_0 \bs G/ \Gamma \to \R$ satisfying
\[
        E(z,s,\frakK) G_k(2s) = \left( 
        \underset{s=1}{\mathrm{Res}} \, \zeta_k(s, \frakK) G_k(s)
        \right)
        \Big[
        \frac{0.5}{s-1} + c_{E,\frakK}(z)
        \Big] + O(|s-1|)
\]
where $\zeta(s, \frakK):= \sum_{I \in \frakK} \frac{1}{\mathrm{Nm}(I)^s  }$.
Consequently, for some $c_{-1,E}>0$ independent of $z$ and some continuous function $c_{0,E,\frakK}(z):  K_0 \bs G/ \Gamma \to \R$, we have
\[
     E(z,s,\frakK) = \frac{0.5c_{-1,E}}{s-1} + c_{0,E,\frakK}(z) + O(|s-1|).
\]
\end{thm}

\subsubsection{Primitive Eisenstein series}
What we need, though, is the behaviour of $E^*(z,s,\frakK)$. For this, note that from the definition,
\begin{equation*}
\begin{aligned}
           E(z, s, \frakK) =  \sum_{\mathfrak{L} \in \mathrm{Cl}(\calO_k)} \zeta(2s, \mathfrak{L}^{-1} \frakK) E^*(z, s,\mathfrak{L} ).
\end{aligned}
\end{equation*}
Let $h:= \normm{    \mathrm{Cl}(\calO_k)
}$ and enumerate $\mathrm{Cl}(\calO_k) = \{ \frakK_1,..., \frakK_h\}$.
Let $M(s)$ be the $h$-by-$h$ matrix whose $(i,j)$-th entry is
\[
     M_{ij}(s) :=  \zeta(2s, \frakK_j^{-1} \frakK_i).
\]
Let 
\[
    \vec{E}(z,s)  := \left( E(z, s, \frakK_1), ....,   E(z, s, \frakK_h) \right)^{\tr},\quad
     \vec{E^*}(z,s):= \left( E^*(z, s, \frakK_1), ....,   E^*(z, s, \frakK_h) \right)^{\tr}.
\]
Then 
\[
       \vec{E}(z,s) = M(s)  \cdot   \vec{E^*}(z,s) \implies \vec{E^*}(z,s) = M(s)^{-1} \cdot \vec{E}(z,s).
\]
Indeed, $M(s)$ is invertible for $s$ close to $1$. To see this, let $\mathrm{Cl}(\calO_k)^{\vee}: = \Hom(\mathrm{Cl}(\calO_k), S^1
)$ be the abelian group of characters where $S^1$ denotes the unit circle in $\C$. The cardinality of $\mathrm{Cl}(\calO_k)^{\vee}$ is exactly $h$, enumerated as $(\chi_1,...,\chi_h)$, and they form an ortho-normal basis of complex functions on $\mathrm{Cl}(\calO_k)$ under the Hermitian form
\[
        \la \chi, \theta \ra := \frac{1}{h} \sum_{i=1}^h \chi(\frakK_i) \overline{\theta(\frakK_i)}.
\]
 For a character $\chi \in \mathrm{Cl}(\calO_k)^{\vee}$, 
\begin{equation*}
\begin{aligned}
      L(s,\chi) := \sum_{\mathfrak{a} \ideal \calO_k} \frac{\chi([\mathfrak{a}])
      }{ \mathrm{Nm}( \mathfrak{a} )^s },\quad
      \quad \vec{v}_{\chi}:=
      \left(
       \chi(\frakK_1),..., \chi(\frakK_h)
      \right)^{\tr}.
\end{aligned}
\end{equation*}
Note that 
\[
    L( s,\chi) = \sum_{i=1}^h  \psi(\frakK_i) \zeta(s, \frakK_i),
\]
and the matrix $V:= ( \vec{v}_{\chi_1},...,\vec{v}_{\chi_h} )$ is unitary:
\[
     (V^{-1})_{ij} = \overline{V_{ji}} = \overline{ \chi_i (\frakK_j)} = \chi_i(\frakK_j)^{-1}.
\]
Therefore,
\begin{equation*}
\begin{aligned}
        M(s) \cdot \vec{v}_{\chi}= L(2s, \chi^{-1})  \vec{v}_{\chi} .
\end{aligned}
\end{equation*}
The eigenvectors $(\vec{v}_{\chi_i} )$ with eigenvalues $( L(2s, \chi_i^{-1}) )$ form a basis, implying that
\[
         M(s) = V \cdot \diag \big(
        L(2s, \chi^{-1}_1)  , ...,  L(2s, \chi^{-1}_h)
        \big)  \cdot V^{-1}.
\]
Hence the $(i,j)$-th entry of $M(s)^{-1}$ is given by
\[
     (M(s)^{-1})_{ij} = \sum_{\lambda =1}^h  \frac{
     \chi_{\lambda}(\frakK_ i)
     }{   L(2s, \chi^{-1}_{\lambda} )
     } \chi_{\lambda}(\frakK_j)^{-1}.
\]
and therefore,
\[
        E^*(z,s, \frakK_i) = \sum_{j,\lambda}   \frac{
     \chi_{\lambda}(\frakK_ i)
     }{   L(2s, \chi^{-1}_{\lambda} )
     } \chi_{\lambda}(\frakK_j)^{-1} E(z,s,\frakK_j).
\]
Plugging into Theorem \ref{theorem_Eisenstein_series}, let
\begin{equation*}
\begin{aligned}
    c_{-1,E^*} &:= \sum_{j,\lambda}   \frac{
     \chi_{\lambda}(\frakK_ i)
     }{   L(2, \chi^{-1}_{\lambda} )
     } \chi_{\lambda}(\frakK_j)^{-1} {c_{-1,E}}
     =  \frac{c_{-1,E} }{\zeta_k(2) },
     \\
     c_{0,E^*,\frakK_i}(z)
     &:=
      \sum_{j,\lambda}   \frac{
     \chi_{\lambda}(\frakK_ i)
     }{   L(2s, \chi^{-1}_{\lambda} )
     } \chi_{\lambda}(\frakK_j)^{-1} c_{0,E,\frakK_i}(z).
\end{aligned}
\end{equation*}
Then
\begin{equation}\label{equation_primitive_Eisenstein_series}
\begin{aligned}
         E^*(z,s,\frakK) =\frac{0.5 c_{-1,E^*}}{s-1} + c_{0,E^*,\frakK}(z) + O(|s-1|).
\end{aligned}
\end{equation}

\subsection{Conclusion}

Identify $K_0 \bs G \cong \bmh_{k_{\infty}}$ by $K_0g\mapsto g^{-1}.o$.
If we write $g= k h^{\Delta}_{x^{-1/l}} u$ with $k\in K_0$, $x\in \R^+$ and $u\in H^{(1)}U$, then by Eq.(\ref{definition_Ny})
\begin{equation*}
    \bmN y(g^{-1}.o) = \bmN y(h^{\Delta}_{x^{1/l}} .o ).
    =
    (x^{2/l})^{l_0} \cdot (x^{2/l})^{ 2l_1} = x^2. 
 \end{equation*}
So we find that $x = \bmN y(g^{-1}.o) ^{1/2}$. Therefore,
\begin{equation*}
\begin{aligned}
         \calG(x^{s+2}) = \sum_{  [\gamma] \in \Gamma/\Gamma \cap H^{(1)}U} \bmN y ( (g\gamma)^{-1}.o )^{\frac{s}{2}+1}
        = \sum_{[\gamma ] \in \Gamma \cap H^{(1)}U \bs  \Gamma} \bmN y ( \gamma g^{-1}.o )^{\frac{s}{2}+1}.
\end{aligned}
\end{equation*}
Note that $\gamma \mapsto (0,1) \cdot \gamma$ induces a bijection
\[
      \Gamma \cap H^{(1)}U \bs  \Gamma \cong \mathrm{Equ}^*(\calO_k).
\]
So
\begin{equation*}
\begin{aligned}
         \calG(x^{s+2}) = \sum_{ \{\alpha,\beta\} \in \mathrm{Equ}^*(\calO_k)
         } \bmN y ( \alpha, \beta; g^{-1}.o )^{\frac{s}{2}+1}
        = E^*(g^{-1}.o, \frac{s}{2}+1, \calO_k).
\end{aligned}
\end{equation*}
Finally,
\begin{equation*}
\begin{aligned}
    \scrD^+(\varphi) &= 
    \lim_{s\to 0^+} 
    \frac{- 1}{ls} \int \varphi \, \diff \rmm_Y + \frac{  c_G^+ c_{-1,E^*}
    }{ s } \int \varphi \, \diff \rmm_Y + c_G^+\int_{Y}
    \varphi(y) c_{0,E^*,\calO_k}(z_y) \,\diff \rmm_Y(y)
    \\
    &=  c_G^+\int_{Y}
    \varphi(y) c_{0,E^*,\calO_k}(z_y) \,\diff \rmm_Y(y)
\end{aligned}
\end{equation*}
where $z_y:= g_{y}^{-1}.o$ if $y= g_y\Gamma$.
Note that the first two terms must cancel since we already knew the limit exists. 
The proof is now complete.

\subsection*{Acknowledgements}
We benefit from discussions with Zhizhong Huang, Fugang Yan and Shucheng Yu. 
I also thank Nimish Shah for 
the $KHU$ decomposition (Lemma \ref{lemma_KAU_decomposition_1}), communicated to the author  several years ago.
The reference \cite{Shi19} was pointed out to us by Zhizhong Huang. 
I am grateful to Hee Oh for her feedbacks on a previous version of the paper.
The author is supported by National Natural Science Foundation of China (No. 12201013).

\bibliographystyle{amsalpha}
\bibliography{ref}

\end{document}